\documentclass[11pt]{article}
\usepackage{latexsym,amssymb,amsmath,amsfonts,amsthm,enumitem}
\usepackage[usenames,dvipsnames]{color}
\definecolor{purple}{rgb}{0.9,0,0.8}

\newtheorem{theorem}{Theorem}
\newtheorem{lemma}{Lemma}

\newtheorem{cor}{Corollary}
\newtheorem{definition}{Definition}

\def\1{\mathbf{1}}
\def\Z{\mathbb{Z}}
\def\da{\downarrow}
\def\ua{\uparrow}
\def\lng{\langle}
\def\rng{\rangle}
\def\beq{ \begin{equation} }
\def\eeq{ \end{equation} }
\def\mn{\medskip\noindent}

\def\ep{\epsilon}
\def\eps{\epsilon}
\def\dlt{\delta}
\def\tld{\tilde}

\def\square{\vcenter{\vbox{\hrule height .4pt
  \hbox{\vrule width .4pt height 5pt \kern 5pt
        \vrule width .4pt} \hrule height .4pt}}}

\def\RR{\mathbb{R}}
\def\R{\mathbb{R}}

\def\ZZ{\mathbb{Z}}
\def\diag{\text{Diag}}
\def\I{\mathcal{I}}
\def\F{\mathcal{F}}

\def\J{\mathcal{J}}
\def\K{\mathcal{K}}
\def\V{\mathcal{V}}

\def\sqz{\kern -0.2em}
\def\var{\hbox{Var}\,}

\def\clearp{}
\def\ctl{\mathtt{ctrl}}
\def\sgn{\text{sgn}}
\def\geo{\text{geometric}}

\usepackage{graphicx}
\graphicspath{%
    {converted_graphics/}
    {/}
}
\begin{document}

\title{Diffusion limit for the partner model at the critical value}
 \author{Anirban Basak, Rick Durrett, and Eric Foxall}

\date{\today}

\maketitle

\begin{abstract} The partner model is an SIS epidemic in a population with random formation and dissolution of partnerships, and with disease transmission only occuring within partnerships. Foxall, Edwards, and van den Driessche \cite{FED} found the critical value and studied the subcritical and supercritical regimes. Recently Foxall \cite{Fox} has shown that (if there are enough initial infecteds $I_0$) the extinction time in the critical model is of order $\sqrt{N}$. Here we improve that result by proving the convergence of $i_N(t)=I(\sqrt{N}t)/\sqrt{N}$ to a limiting diffusion. We do this by showing that within a short time, this four dimensional process collapses to two dimensions: the number of $SI$ and $II$ partnerships are constant multiples of the the number of infected singles. The other variable, the total number of singles, fluctuates around its equilibrium like an Ornstein-Uhlenbeck process of magnitude $\sqrt{N}$ on the original time scale and averages out of the limit theorem for $i_N(t)$. As a by-product of our proof we show that if $\tau_N$ is the extinction time of $i_N(t)$ (on the $\sqrt{N}$ time scale) then $\tau_N$ has a limit.
\end{abstract}

\section{Introduction}

In the partner model each of $N$ individuals can be susceptible or infected and in a partnership or not. So the system is described
by the five quantities $S_t$ and $I_t$, the number of single susceptible and infected individuals, and $SS_t$, $SI_t$, and $II_t$, the number of partnered pairs of the three possible combinations, at time $t$.
Infected individuals become healthy (and susceptible to re-infection) at rate 1. A susceptible individual with an infected partner
becomes infected at rate $\lambda$. Partnerships dissolve at rate $r_-$. Each pair of single individuals forms a partnership at rate $r_+/N$.

Foxall, Edwards, and van den Driessche \cite{FED} introduced this model and showed that despite the complexity of the model it is possible to find the critical value explicitly. To do this they used the continuous time Markov chain $X_t$ with state space $\{A,B,C,D,E,F,G\}$ and rates as shown in Figure \ref{fig:critval}. Thinking of a single infected individual in an otherwise susceptible population, we start in state $A$, and let $\tau$ be the first time $X_t$ enters $\{ D, E, F, G \}$. The basic reproduction number for the model is
\beq
R_0 = P_A(X_\tau = F) + 2 P_A(X_\tau = G)
\label{r0eq}
\eeq
which is the expected number of infected singles at time $\tau$. The critical value of $\lambda$ is
$$
\lambda_c = \sup\{ \lambda \ge 0 , R_0 \le 1\}
$$
with $\lambda_c < \infty$ if and only if $r_+ > 1 + 1/r_-$. There is an explicit formula for $\lambda_c$ but it is not very pretty since the
formulas for the hitting probabilities are somewhat complicated.

In \cite{FED} it was shown that

\begin{theorem} \label{FEDth}
If $R_0 < 1$ there are constants $T$, $C$ so that, from any initial configuration, with high probability the process dies out by time $T+ C\log N$. If $R_0>1$, then for any $\epsilon>0$ there are constants $T$, $C$, $\gamma$, such that from any initial configuration with at least $\ep N$ infected, with probability at least $1- e^{-\gamma N}$ the process survives for time $e^{\gamma N}$ and the frequencies of the five types $s_t=S_t/N, i_t=I_t/N,$ etc. are within $\ep$ of their equilibrium values $(s_*,i_*,ss_*,si_*,ii_*)$ when $T \le t \le e^{\gamma N}$.
\end{theorem}

\begin{figure}[t]
\begin{center}
\begin{picture}(280,140)
\put(30,100){$A=I$}
\put(140,100){$B=SI$}
\put(250,100){$C=II$}
\put(30,30){$D=S$}
\put(100,30){$E=S,S$}
\put(180,30){$F=S,I$}
\put(250,30){$G=I,I$}
\put(70,105){\vector(1,0){60}}
\put(85,90){$r_+y_*$}
\put(190,105){\vector(1,0){50}}
\put(210,115){$\lambda$}
\put(240,100){\vector(-1,0){50}}
\put(210,85){2}
\put(50,90){\vector(0,-1){40}}
\put(30,70){1}
\put(270,90){\vector(0,-1){40}}
\put(280,70){$r_-$}
\put(150,90){\vector(-3,-4){30}}
\put(115,60){1}
\put(170,90){\vector(3,-4){30}}
\put(205,60){$r_-$}
\end{picture}
\caption{Markov chain for computation of $R_0$}
\end{center}
\label{fig:critval}
\end{figure}
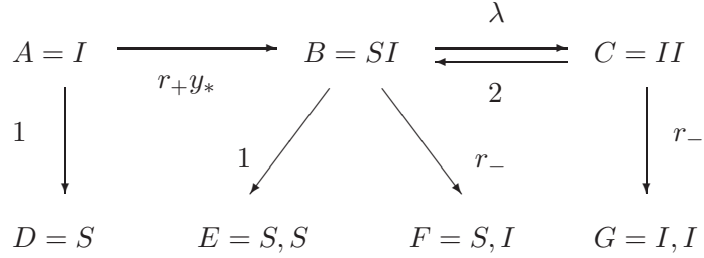

\noindent
To describe the equilibrium values we need some notation.
Let $Y_t = S_t+I_t$ be the number of single individuals and $y_t = Y_t/N$. $y_t$ approaches and remains close to a stationary value $y_*$
which is the unique equilibrium in $(0,1)$ for the ODE
$$
y' = r_-(1-y) - r_+ y^2
\label{eq:yeq}
$$
We will explain this result in more detail later, see \eqref{ystar}. The equilibrium frequency of singles, $y_*$,
is the solution of $r_-(1-y_*) = r_+ y^2_*$.

To find the number of single infecteds in equilibrium we let $i_t = I_t/N$ and note the three events that affect the number of infected singles are

\begin{itemize}
  \item $I \to S$ at rate $I_t = i_tN$,
  \item $I+I\to II$ at rate $(r_+/N)\binom{I_t}{2} \approx r_+ (i_t^2/2)N$, and
  \item $S + I \to SI$ at rate $(r_+/N)S_tI_t \approx r_+ i_t s_tN$.
\end{itemize}

\noindent
Fixing $i_t = i \in (0,y^*)$ we define probabilities
$$
p_S = \frac{1}{z} \qquad p_{II} = \frac{r_+ i}{2z} \qquad p_{SI} = \frac{r_+(y_*-i)}{z}
$$
where $z=1+r_+(y_*- i/2)$ is the sum of the numerators. Let
\begin{align*}
\Delta_S = -1 \qquad \Delta_{II} & = - 2 + P_C(X_\tau = F) + 2 P_C(X_\tau = G) \\
 \Delta_{SI} &= -1 + P_B(X_\tau = F) + 2 P_B(X_\tau=G)
\end{align*}
be the expected change in the number of single infecteds (at partnership breakup) due to the three events. Finally let
$$
\Delta(i) = p_S \Delta_S + p_{II}\Delta_{II} + p_{SI}\Delta_{SI}
$$
be the expected change in the number of infecteds per event. In equilibrium $\Delta(i_*)=0$. Having found $i_*$ and $s_* = y_* - i_*$,
it is routine to find $ii_*$, $si_*$, and $ss_*$; see Section 5 of \cite{FED} for more details. It's also worth noting that the condition $i_*=0$ is equivalent to $R_0=1$. \\

The analysis of the critical case was done in a second paper
by Foxall \cite{Fox}. The main result is

\begin{theorem} \label{Foxth}
Let $V_t$ be the number of infected vertices at time $t$. If $R_0=1$ then
\begin{itemize}
  \item there are $C,\gamma > 0$ so that from any initial configuration, with probability at least $ 1 -e^{-\gamma m}$, $V_{mC\sqrt{N}} = 0$, and
  \item if $V_0 \ge \sqrt{N}$ and $y_0 \ge y_* -(\log N)/\sqrt{N}$ there is $c>0$ so that $V_{c\sqrt{N}} \neq 0$
with probability at least $1 - e^{-c(\log N)^2}$.
\end{itemize}
\end{theorem}

\noindent
The goal of this paper is to obtain a more complete description of the process in the case $R_0=1$. In particular, for the rest of the paper we assume that $R_0=1$.\\

Theorem \ref{Foxth} shows that the extinction time $\tau_N$ is of order $N^{1/2}$ if $V_0 \ge \sqrt{N}$. By analogy with critical branching processes one might expect the time to be of order $N$ if $V_0 = N$ (same order as the initial values). To explain why $N^{1/2}$ is the right order of magnitude and to indicate what more precise result we would like to prove, we sketch the proof in the following simpler setting.

\mn
{\bf Example: Contact process on a complete graph with $N$ vertices.} Individuals die at rate 1, and give birth at rate $\beta$ to an offspring that is sent to a randomly chosen vertex, so the number of occupied vertices $X(t)$ is a Markov chain on $\{0, 1, \ldots N \}$ with transition rates
$$
q(k,k-1) = k \quad \hbox{and} \quad q(k,k+1) = \beta k (1 - k/N).
$$
The critical value for prolonged survival is $\beta_c=1$. 

\begin{theorem}\label{th:MFCP}
Let $x_N(t) = X(N^{1/2}t)/N^{1/2}$, with $\beta=1$. Then $x_N(t) \Rightarrow x_t$, the solution of
\beq
dx_t = - x_t^2 \, dt + \sqrt{2x_t} dB_t.
\label{MFCP}
\eeq
Let $\tau_0(x_N) = \inf\{t : x_N(t) = 0 \}$. If $x_N(0) \to \infty$, $\tau_0(x_N) \Rightarrow \tau_0(x)$, the hitting time of 0 for the diffusion process started at $\infty$.

\end{theorem}

\begin{proof}
With $\beta=1$ we find
$$\frac{d}{dt}E[X_t] = -(E[X_t^2])/N$$
and using Jensen's inequality, $E[X_t]$ satisfies the differential inequality $y' \le -y^2/N$. Since $E[X_0] \le N$ this gives $E[X_t] \le N/(1+t)$, and using Markov's inequality $P(X_{\ep N^{1/2}} \le \ep^{-2} N^{1/2}) \ge 1-\ep$. Letting Let $x_N(t) = X(N^{1/2}t)/N^{1/2}$ this means that $x_N(\ep) \le \ep^{-2}$ with probability at least $1-\ep$. The drift of $x^N_t$ is
$$
N^{1/2} \, \frac{1}{N^{1/2}} \, [ - x_N(t) N^{1/2} + x_N(t) N^{1/2} (1-x_N(t)/N^{1/2}) ]
= -(x_N(t))^2
$$
while the diffusivity is
$$
N^{1/2} \, \frac{1}{N} \, [ x_N(t) N^{1/2} + x_N(t) N^{1/2} (1-x_N(t)/N^{1/2}) ]
\approx 2 x_N(t)
$$
(these terms are explained in Section \ref{sec:sampath}). The first result then follows from Lemma \ref{lem:limproc} in the next section. Then, as in the proof of Theorem \ref{T0lim} in Section \ref{subsec:extime}, to show $\tau_0(x_N) \Rightarrow \tau_0(x)$, it is enough to show that for each $\ep>0$ we can find $\dlt>0$, so that if $x_N(t) \le \dlt$ then $x_N(t+\ep)=0$ with probability at least $1-\ep$. This is easy to do once we note that $X_t$ is dominated by the critical branching process $\tilde X_t$ in which each particle splits in two, or dies, each at rate one. We know that $P(\tilde X_t=0 \mid \tilde X_0 = k) = (1-(1+t)^{-1})^k \ge 1-k/t$, which is at least $1-\ep$ if we let $t = k/\ep$. Letting $k=\delta N^{1/2}$, the result follows with $\delta=\ep^2$.\\
\end{proof}

In the above example, there are three main steps:
\begin{enumerate}[noitemsep,label={\roman*)}]
 \item Show $X_t$ comes down to $C_\ep \sqrt{N}$ within $\ep N^{1/2}$ time,
 \item Show that $x_N(t) = X_{N^{1/2}t}/N^{1/2}$ converges to a diffusion,
 \item Show that once $x_N(t)$ is small, it hits zero in a short time.
\end{enumerate}

\medskip

The corresponding result for the partner model follows the same three steps, but is more complicated because the process is four dimensional and there are two different time scales.\\

\noindent \textbf{Some notation.} To state our results and to avoid confusion between $SI$ and $S \cdot I$ etc we introduce alternative notations that we will use throughout the paper: $J=II$, $K=SI$ and $L=SS$. We refer to the stochastic process $(S,I,J,K,L)$ as the \emph{infection process}. We outline some further notational conventions below.\\

\noindent\textit{Asymptotic notation.} Let $a_N,b_N$ be sequences of real numbers.
\begin{itemize}[noitemsep]
\item $a_N = O(b_N)$ if $\limsup_{N\to\infty}|a_N/b_N| < \infty$.
\item $a_N = \Omega(b_N)$ if $\liminf_{N\to\infty}|a_N/b_N|>0$.
\item $a_N = o(b_N)$ if $\lim_{N\to\infty}|a_N/b_N| = 0$.
\item $a_N = \omega(b_N)$ if $\lim_{N\to\infty}|a_N/b_N| = \infty$.
\end{itemize}
Moreover, for efficiency of notation we will say that a certain property holds for $o(f(N)) \le t \le \omega(f(N))$ if there exist sequences $a_N,b_N$ with $a_N=o(f(N))$ and $b_N = \omega(f(N))$ such that the property holds for all $t \in [a_N,b_N]$.\\

\noindent \textit{Rescaling.} Throughout the paper, the placement of the time variable, as for example $I_t$ or $I(t)$, is chosen according to notational convenience and does not change the meaning. On the other hand, we will often want to rescale in either time or space, so we introduce the following notation for these purposes. When we need to distinguish the spatial scale, upper case is reserved for the originally defined variables $S,I,J,K,L$ and $Y=S+I$, and lower case denotes the following:
\begin{center}
$s_t^N = S_t/N$, $y_t^N = Y_t/N$, \\
\vspace{10pt}
$i_t^N = I_t/\sqrt{N}$, $j_t^N = J_t/\sqrt{N}$, and $k_t^N = K_t/\sqrt{N}$.
\end{center}
Note that $I,J,K$ are rescaled by $1/\sqrt{N}$, and not by $1/N$ as in Theorem \ref{FEDth}; as demonstrated by Theorem \ref{th:MFCP} this rescaling is more appropriate when $R_0=1$. To distinguish time scales we note that two scales will be relevant: the original time scale that we call fast, and the $\sqrt{N}$ time scale that we call slow, or long. We will use the superscript $^N$ for the fast time scale and the subscript $_N$ for the slow time scale. So, for example,
\begin{center}
$i_t^N = I_t/\sqrt{N}$ and $i_N(t)=I(\sqrt{N}t)/\sqrt{N}$.
\end{center}
This distinction appears in discussion as well as computation: for example, saying that $i_N$ reaches a certain value within $O(1)$ amount of time is the same as saying that $i^N$ reaches that value in $O(N^{1/2})$ time. 
For the sake of consistency, and to distinguish from limit processes, we will write $S^N, I^N$, etc. for the originally defined variables, on the original time scale. A few more processes will be introduced later, and they will follow the same notation for time scale; some will have upper and lower case versions, again to denote different spatial scales; the precise scaling is specified in each case.\\

\noindent \textit{Times.} We will use lower-case $s$ or $t$ for time variables, $T$ to denote a fixed (deterministic) time and $\tau$ for stopping times. For an $\R$-valued process $X$, we let
$$\tau^-_x(X,t) = \inf\{s \ge t \colon X_s \le x \} \quad \text{and} \quad \tau_y^+(X,t) = \inf\{s\ge t\colon X_s>y\},$$
and if $X$ is $\R_+$-valued we write $\tau_0(X,t)$ for $\tau_0^-(X,t)$. $\tau^-_x(X)$ denotes $\tau^-_x(X,0)$ and similarly in other cases. In one case we will need the following:
$$\tau_{x,y}(X,t) = \inf\{s \ge t \colon X_t \le x \ \text{or} \ X_t > y\}.$$
We will also define some labelled and unlabelled times such as $\tau,\tau^*,\tau_1,\tau_2,\dots$ when proving specific results, if they do not fit the above template, or to save on notation if they are written frequently.\\

We first describe the limit processes, followed by statements of the main results, and then we provide the workflow. Throughout the paper, if we say a statement holds with high probability (whp), then it has probability tending to $1$ as $N\to\infty$.

\bigskip

\noindent\textbf{Deterministic limits.} We show in Section \ref{sec:O1} that as $N\to \infty$, sample paths of $y_t^N$ and $(i_t^N,j_t^N,k_t^N)$ converge in distribution to deterministic limits $y_t$ and $(i_t,j_t,k_t)$. 
We find that $\lim_{t\to\infty}y_t = y_*$, the solution of $r_+y_*^2 = (1-y_*) r_-$, while $(i_t,j_t,k_t)$ converges as $t\to\infty$ to a point on the ray $(\alpha ,\beta , 1)\RR_+$
of fixed points for the linear system described in \eqref{ijk}, where $\alpha,\beta$ are given by \eqref{eq:invray}. 
This is reminiscent of (multiplicative) state space collapse in queueing networks where a vector
of queue lengths are all proportional to one of them. There are many results of this type.
For examples, see \cite{Harrison, Bramson, RJW, Stolyar, ShWi}.\\

\noindent\textbf{Diffusion limits.} We will show that the fluctuations $z_t^N = \sqrt{N}(y_t^N-y_*)$ are approximately an Ornstein-Uhlenbeck process
$$
dz = - \mu_z z dt + \sigma_z dB
$$
that evolves on the fast time scale, where $\mu_z,\sigma_z$ are some positive constants. On the other hand, the fluctuations of $(i^N,j^N,k^N)$ (once they are close to the ray) occur on the slow time scale. Since it stays close to a ray in phase space, $(i^N,j^N,k^N)$ is effectively one-dimensional, and we will show that $i_N(t)= I^N(\sqrt{N} t)/\sqrt{N}$ converges to the limit in \eqref{MFCP} but with different constants for the mean and variance. As in the previous result, the hitting time of $0$ converges. More precisely, we prove the following results, which are the main goal of this article. As pointed out earlier, $R_0=1$ is assumed throughout; for clarity, we recall this assumption in each of our main results.

\begin{theorem} \label{finiv}
Suppose that $R_0=1$, $|(i_N(0),j_N(0),k_N(0)) - (\alpha x,\beta x, x)| = O(N^{-\eps})$ for some $x,\eps>0$ and $|z_N(0)| = O(1)$ as
$N \to \infty$, where $\alpha$ and $\beta$ are defined in \eqref{eq:invray}. Then there are constants $\mu_X,\sigma^2_X>0$ such that for any fixed $T>0$, $i_N$ converges in distribution in $C[0,T]$ to the diffusion
\begin{align}
dX_t = -\mu_X X_t^2 dt + \sigma_X \sqrt{X_t} dB_t
\label{eq:lim-diff}
\end{align}
started from $X_0=\alpha x$. 
\end{theorem}

\vskip10pt

It is possible to compute the constants $\mu_X,\sigma_X$ from our proof; since the expressions are not particularly nice-looking and do not add much insight, we have omitted them. It is possible the assumptions on the rate of convergence of initial data may be relaxed; to do so one would require a more careful account of transient behavior. Since this paper is already lengthy, we have not pursued this extension.\\

Our next result builds on Theorem \ref{finiv} and allows the process to start from $\infty$. In the proof of Theorem \ref{infiv}, which is in Section \ref{sec:workflow}, we also point out why it makes sense to start the limiting diffusion \eqref{eq:lim-diff} from $\infty$.

\begin{theorem} \label{infiv}
Suppose that $R_0=1$, $i_N(0)+j_N(0)+k_N(0) \to \infty$ and $y_N(0) \to y_*$ as
$N \to \infty$. Then $i_N$ converges in distribution in $C[\ep,T]$ for any fixed $0<\ep<T<\infty$ to the diffusion \eqref{eq:lim-diff},
started from $X_0=\infty$.
\end{theorem}

\vskip10pt

It is possible that in Theorem \ref{infiv} the assumption $y_N(0)\to y_*$ can be dropped. We do not pursue this direction in this paper. 
Lastly we show convergence of the hitting time.

\begin{theorem} \label{T0lim}
Under the assumptions of Theorem \ref{finiv} or \ref{infiv}, for
$$\tau_0(i_N+j_N+k_N) = \inf\{t:(i_N(t),j_N(t),k_N(t))=(0,0,0)\}$$
we have $\tau_0(i_N+j_N+k_N) \Rightarrow \tau_0(X)$,
the time to hit zero for the limiting diffusion \eqref{eq:lim-diff}.
\end{theorem}

\vskip10pt

 The separation of time scales between $y^N$ and the infection variables $i^N,j^N,k^N$ may remind the reader of the work of Kang and Kurtz \cite{KK} and Kang, Kurtz, and Popovic \cite{KKP} on chemical reaction networks. We found that writing our model in their framework does not simplify the difficult aspects of the proof, so we have opted instead to use a general result of \cite{EthierKurtz} for obtaining the diffusion limit, which is tailored to our context in Lemma \ref{lem:limproc}.\\

\noindent\textbf{Workflow.} There are seven main steps, described in greater detail in Section \ref{sec:workflow}. 
In order to make certain estimates it is helpful to define the following additional observables:
\begin{enumerate}[noitemsep,label={\roman*)}]
 \item The positive linear combination $H^N_t=I^N_t+\gamma J^N_t+\eta K^N_t$ and $h^N_t = H^N_t/\sqrt{N}$, where $(1,\gamma,\eta)$, defined by \eqref{leftev}, is a left eigenvector for the matrix $A$ given by \eqref{Adef}, that determines the linear system \eqref{ijk} for $(i_t,j_t,k_t)$. The variable $H^N$ is helpful to the analysis because the linear terms drop out of the equation for its drift.
 \item The rescaled infection vector $(U^N,V^N,W^N)=(I^N,\gamma J^N, \eta K^N)/H^N$, and its metastable equilibrium value $(u_*,v_*,w_*)$ which is given by \eqref{eq:uvw-lim}.
 \item The quantity $Q^N = \theta_2(U^N-u_*)^2 + \theta_1(V^N-v_*)^2$, where $\theta_1,\theta_2$, given by \eqref{eq:thetas}, are well-chosen positive constants. $Q^N$ measures the deviation of $(U^N,V^N,W^N)$ from equilibrium, so when $Q^N$ is small, $(i_t^N,j_t^N,k_t^N)$ is close to the invariant ray, a property which is essential to obtaining a limiting 1-dimensional equation for the diffusion.
\end{enumerate}

Of course, $H_N(t)=H^N(\sqrt{N}t)$, etc. The main steps, written in terms of the slow time scale, are sketched below in the context of Theorem \ref{infiv}, when $h_N(0)=\omega(1)$; for Theorem \ref{finiv}, we just need that if $|z_N(0)|,h_N(0)$ and $Q_N(0)$ are small, then they can be kept small for $\omega(1)$ amount of time.
\begin{enumerate}[noitemsep,label={\arabic*.}]
 \item \label{step:z} \textbf{Lemma \ref{maxz}.} Show that $|z_N(t)| = O(\sqrt{\log N})$ for $o(1) \le t \le \omega(1)$.
 \item \label{step:h1} \textbf{Lemma \ref{bdI}.} Show that $h_N(t) = O(\log N)$ for $o(1) \le t \le \omega(1)$.
 \item \label{step:ray} \textbf{Lemma \ref{ray}.} Show that $|Q_N(t)| = O(N^{-1/6})$ for $o(1) \le t \le \omega(1)$.
 \item \label{step:avg} \textbf{Lemma \ref{aveto0}.} Show that the integral of $i_N(t)z_N(t)$ averages to zero on finite time intervals.
 \item \label{step:h2} \textbf{Lemma \ref{downfast}.} Show for any $\ep>0$ there is $C_\ep>0$ so that $P(\tau_{C_\ep}^-(h_N) \le \ep)\ge 1-\ep$.
 \item \label{step:diff} \textbf{Lemma \ref{lem:Hdown}, Theorems \ref{finiv} and \ref{infiv}.} Show convergence of $i_N(t)$ to the diffusion limit.
 \item \label{step:down} \textbf{Lemma \ref{lem:Hdn1}.} Show that $\tau_0(h_N) \to 0$ in probability as $h_N(0) \to 0$, uniformly for large $N$.
\end{enumerate}

The organization of the rest of the paper is as follows: In Section \ref{sec:sampath} we gather some probability estimates and limit theorems that are used throughout the paper. In Section \ref{sec:O1} we describe the deterministic limits. In Section \ref{sec:workflow} we state precise lemmas relating to each of the workflow steps, and prove diffusion limits. The latter sections are devoted to proofs of the lemmas.

\section{Sample path estimation}\label{sec:sampath}
 
In this section we describe some sample path estimates, a diffusion limit theorem, and results on drift and diffusivity of functions of  continuous time Markov chains that are used throughout the paper. Any results that are not cited are proved in the Appendix. The natural setting for these results is semimartingales. For an overview of the semimartingale theory that is used here, we refer the reader to \cite{jacod}.\\

First we recall some standard definitions from semimartingale theory, noting along the way how the present context fits into this framework.\\

Let $(\Omega,\F,\bf F,P)$ be a filtered probability space that satisfy the usual conditions. This means that the filtration $\mathbf{F} = (\F_t)_{t \in \R_+}$ is right continuous in the sense that $\F_t = \bigwedge_{s>t}\F_s$ for each $t$, and each $\F_t$ in the filtration $\mathbf{F} = (\F_t)_{t \in \R_+}$ contains the $\bf P$ null sets of $\F$. In \cite{jacod} it is also assumed that $\F$ is $\bf P$ complete; if this is not the case then it is easy to check that completing $\F$ and then $\bf F$ with respect to null sets does not violate right continuity. In our case, the filtered probability space is that of the finite state continuous-time Markov chain corresponding to the state variables $(S^N,I^N,J^N,K^N,L^N)$, with the completion of the natural filtration. Since such a process is Feller, as shown in \cite[I.5]{protter}, the corresponding filtered space satisfies the usual conditions.\\

A procss $X$ is called \emph{optional} if it is measurable with respect to the $\sigma$-field (on $\Omega \times \R_+$) generated by all c\`adl\`ag adapted processes. All the processes considered here are optional. We assume the reader is familiar with the notions of stopping time, predictable time and process, localization and martingale.\\

Given a stochastic process $X$ we denote by $X_-$ the left-continuous process obtained from $X$. We further let $\Delta X= X- X_-$ denote the process of jumps. We say that $X$ has \emph{bounded jumps} if $|\Delta X| \le c$ a.s.~for some constant $c>0$, and let $\Delta_\star(X)$ denote the infimum of such values of $c$. $X$ is \emph{quasi-left continuous} (qlc) if $\Delta X_\tau=0$ a.s.~on $\{\tau<\infty\}$ for any predictable time $\tau$.\\

Given a process $A$, define the process $\var(A)$ by setting $\var(A)_t(\omega)$ equal to the total variation of the function $s\mapsto A_s(\omega)$ on the interval $[0,t]$. A process $A$ has \emph{finite variation} if $\var(A)_t(\omega)<\infty$ for each $t,\omega$, and is \emph{locally integrable} if it has a localizing sequence $(\tau_n)$ such that $E[\var(A)_{\tau_n}]<\infty$ for each $n$. The \emph{compensator} of a locally integrable process $A$, denoted $A^p$, is the unique predictable and locally integrable process such that $A-A^p$ is a local martingale (see \cite[I.3.18]{jacod}).\\

A \emph{semimartingale} (s-m) $X$ is a process that can be written as $X=X_0+M+A$, where $X_0$ is an $\F_0$-measurable random variable, $M$ is a local martingale and $A$ has finite variation. We call a semimartingale \emph{special} if it can be written as
\begin{equation}\label{eq:spec-sm}
X = X_0 + X^m + X^p 
\end{equation}
where $X^p$ is the compensator of $X$ and $X^m$ is a uniquely defined local martingale satisfying $X_0^m=0$.
By \cite[I.4.24]{jacod}, if $X$ has bounded jumps then it is special and $|\Delta X^m| \le 2\Delta_\star(X)$, and if it also qlc then using \cite[I.2.35]{jacod} in the proof of \cite[I.4.24]{jacod}, we have the more convenient estimate $\Delta_\star(X^m) \le \Delta_\star(X)$.\\

Recall that if a martingale $M$ is locally square-integrable then $M^2$ has a compensator, denoted $\langle M \rangle$ and called the predictable quadratic variation (pqv). Any local martingale $M$ with $M_0=0$ and bounded jumps is locally square integrable (see \cite[I.4.1]{jacod}). If $X$ is a special s-m and $X^m$ is locally square-integrable we will use $\lng X \rng$ to denote $\lng X^m \rng$. The following basic estimate is used repeatedly throughout the paper.
\begin{lemma}[{\cite[Lemma 3]{naming-game}}]\label{lem:sm-est}
Let $X$ be a quasi-left continuous semimartingale with bounded jumps and let $a,\phi>0$.
\begin{equation}\label{eq:sm-est2}
\text{If} \quad 0<\phi \Delta_\star(X)\le 1/2 \quad \text{then} \quad P\left( \sup_{t\ge 0} \left\{ |X_t^m| - \phi \langle X \rangle_t \right\} \ge a \right) \le 2e^{-\phi a}.
\end{equation}
\end{lemma}

In order to obtain tidy expressions for $X^p$ and $\langle X \rangle$, we define quasi-absolute continuity, together with drift and diffusivity. To motivate the name, note that if $X$ is a special semimartingale with $X^m$ locally square-integrable, then as shown in \cite[Lemma 2]{naming-game}, $X$ is qlc iff both $X^p$ and $\langle X \rangle$ are continuous.

\begin{definition}
A special semimartingale $X$ with locally square-integrable martingale part $X^m$ is called \emph{quasi-absolutely continuous} (qac) if both $X^p$ and $\langle X \rangle$ are absolutely continuous. In this case define the \emph{drift} $\mu(X) = (\mu_t(X))_t$ and the \emph{diffusivity} $\sigma^2(X) = (\sigma_t(X))_t$ for Lebesgue-a.e. $t$ by
\begin{equation}\label{eq:mu-sig}
\mu_t(X) = \frac{d}{dt}X^p_t, \quad \sigma^2_t(X) = \frac{d}{dt}\langle X \rangle_t.
\end{equation}
\end{definition}

Letting $X = (S^N,I^N,J^N,K^N,L^N)$ denote the infection process, we note that the estimates of Lemma \ref{lem:sm-est} and the upcoming Lemma \ref{lem:driftbar} will be applied to observables of the form $f(X_t^\tau)$, where $f:R \to \R$ is Lipschitz, $R \subset \R^5$ is compact, and $\tau$ is a stopping time so that the stopped process $X_t^\tau = X_{t\wedge \tau}$ has $X_t^\tau \in R$ for $t \in \R_+$. Any such process is right continuous and has finite variation so is a s-m, and has bounded jumps since this is the case for $X$ and since $f$ is Lipschitz. Let $q_i(x)$, $i=1,\dots,n$ denote the rates of the various transitions ($S+I \to K, K \to J, J\to K$ etc.) and $\Delta_i$ their effect (for example, $S^N,I^N$ decreases by 1, $K^N$ increases by 1), and let $Y_t = f(X_t^\tau)$. By writing $Y_t$ as a sum of jumps and using the standard linear and quadratic martingales for Poisson processes, it is easy to show that $Y$ is qac (thus also qlc) and has 
\begin{align*}
\mu_t(Y) &= \1(t < \tau)\sum_{i=1}^n q_i(X_t)(f(X_t+\Delta_i)-f(X_t)), \\ 
\sigma^2_t(Y) &= \1(t < \tau)\sum_{i=1}^n q_i(X_t)(f(X_t+\Delta_i)-f(X_t))^2,
\end{align*}
a fact that we use ubiquitously when invoking Lemma \ref{lem:sm-est} below. 
In a few cases, we are interested in computing the drift for products of such processes, as well as processes of the form $f(X_t^\tau)g(t)$, where $g(t)$ is an absolutely continuous deterministic function (which is easily shown to be a qac s-m). 
Below we state a result from \cite{naming-game} that allows to find the drift for products of qac s-m.
\begin{lemma}[Product rule]\label{lem:qac-prod}
Suppose $X_t,Y_t$ are qac semimartingales on a common filtered probability space. Then $\lng X^m,Y^m \rng$ and $(XY)^p$ exist and are absolutely continuous. Denote $\mu_t(XY) = \frac{d}{dt}(XY)^p_t$. Then
$$\mu(XY) = \sigma(X,Y) + X_- \mu(Y) + Y_- \mu(X),$$ 
where $\sigma_t(X,Y) = \frac{d}{dt} \lng X^m,Y^m\rng_t$.
\end{lemma}

The next result will allow us to obtain simple estimates for the drift of certain processes. As in \cite[I.3.4]{jacod}, if $H$ is optional and $A$ has locally finite variation, we let
$$(H\cdot A)_t(\omega) = \int_0^t H_s(\omega)dA_s(\omega),$$
for $\omega \in \Omega$ such that the above can be evaluated as a Stieltjes integral; a sufficient condition is that $H$ also has finite variation. Of course, if $A$ is absolutely continuous then $H\cdot A$ is differentiable and $\frac{d}{dt}(H\cdot A)_t = H_t \, \frac{d}{dt}A_t$ for Lebesgue-a.e.$\,t$. As noted in \cite[I.3.18]{jacod}, if $H$ is predictable and both $A, H\cdot A$ are locally integrable then $H \cdot A^p = (H\cdot A)^p$, which we use below.

\begin{lemma}[Taylor approximation]\label{lem:taylor}
Let $X$ be a qac s-m with bounded jumps and let $f \in C^2(\R)$. Then, $f(X)$ is a qac s-m and satisfies the following inequality for Lebesgue-a.e. $t$:
$$|\mu_t(f(X)) - f'(X_t)\mu_t(X)| \le \frac{1}{2}\sigma_t^2(X)\sup_{|x-X_t| \le \Delta_\star(X)}|f''(x)|.$$
\end{lemma}

Building on Lemma \ref{lem:sm-est} we derive the following simple result, which is the basis of several estimates in this article.

\begin{lemma}[Drift barrier]\label{lem:driftbar}
 Fix $x>0$ and let $X$ be a qac s-m on $\R$ with bounded jumps, such that $\Delta_\star(X) \le x/2$. Suppose there are positive reals $\mu_\star,\sigma^2_\star,C_{\mu_\star},C_\Delta$ with $\max\{\Delta_\star(X)\mu_\star/\sigma^2_\star, 1/2\} \le C_\Delta$ so that if $0<X_t<x$ then
\begin{align}\label{eq:driftbar-est}
\mu_t(X) \leq -\mu_\star,\quad |\mu_t(X)| \leq C_{\mu_\star} \quad\hbox{and}\quad \sigma^2_t(X) \leq \sigma^2_\star.
\end{align}
Let $\Gamma = \exp(\mu_\star x /(32C_\Delta\sigma^2_\star))$. Then we have
 \begin{equation}\label{eq:driftbar}
 P\left( \ \sup_{t \le \lfloor \Gamma \rfloor x/16C_{\mu_\star} }X_t \ge x \ \mid \ X_0 \leq x/2 \ \right) \le 4/\Gamma.
 \end{equation}
\end{lemma}

We will occasionally need to apply Lemma \ref{lem:driftbar} to a stopped process, for which the following easy corollary will be helpful.

\begin{cor}[Drift barrier with stopped process]\label{cor:driftbar}
In the setting of Lemma \ref{lem:driftbar}, let $\tau$ be a stopping time and suppose that \eqref{eq:driftbar-est} holds assuming $t<\tau$ in addition to $0<X_t<x$. Then
 \begin{equation*}
 P\left( \ \sup_{t \le \tau \wedge \lfloor \Gamma \rfloor x/16C_{\mu_\star} }X_t \ge x \ \mid \ X_0 \leq x/2 \ \right) \le 4/\Gamma.
 \end{equation*}
\end{cor}

The final result of this section is about identifying the diffusion limit (or an ODE limit) for a stochastic process. Before stating the result let us introduce some more notations. A stochastic process $X_t = (X_{t,1},\dots,X_{t,d})$ is an $\RR^d$-valued semimartingale  if each component is a s-m. The  compensator (when it exists) of $X_t$ is defined component-wise and the predictable quadratic variation (when exists) is a $d\times d$ matrix-valued process with $\langle X \rangle_{t,ij} = \langle X_i^m, X_j^m \rangle_t$. Drift and diffusivity are similarly defined. Here we use Theorem 4.1 in \cite{EthierKurtz} to obtain easily checkable conditions for convergence to an ODE or diffusion limit.

\begin{lemma}[Diffusion limit]\label{lem:limproc}
Let $X_t^N$ be a sequence of qlc semimartingales, $a$ be a Lipschitz $d\times d$ matrix-valued function on $\R^d$ and $b:\R^d\to \R^d$ be Lipschitz. Suppose that a.s. $|\Delta X_t^N| \le c_N$ with $c_N \to 0$ as $N\to\infty$. Also assume that for each $T,R>0$,
\begin{align}
\sup_{t \le T\cap \tau_R^+(|X^N|)}|(X_i^N)^p - \int_0^t b(X_{i,s}^N)ds| & \to 0, \label{eq:diff-bd-1}\\
\sup_{t \le T\cap \tau_R^+(|X^N|)}|(\langle X_i^N,X_j^N \rangle_t - \int_0^t a_{ij}(XX^{\top}) ds| & \to 0,  \label{eq:diff-bd-2}
\end{align}
as $ N\to\infty$, where the convergence holds in probability, and $\tau_R^+(|X^N|) = \inf\{t:|X_t^N| \ge R\}$. Suppose $X_0^N \to x \in \RR^d$. Then $X_t^N$ converges in distribution to the diffusion process $x$ with
$$x_0=x \quad \text{and}\quad dx = b(x)dt + a(x)dB.$$
In particular, if $a=0$ then $X^N$ converges to the solution of the ODE system with
$$x_0=x \quad \text{and}\quad x' = b(x).$$
The above statements also hold if $X^N$ are qac with drift and diffusivity given by functions $\mu_N,\sigma^2_N$ satisfying
$$\sup_{|x| \le R}|\mu_N(x) - b(x)|,|\sigma_N^2(xx^{\top}) - a(xx^{\top})| \to 0.$$
\end{lemma}

\section{Deterministic limits} \label{sec:O1}

The goal of this section is show that on the fast time scale, $y^N$ and $(i^N,j^N,k^N)$ have deterministic limits, and to compute a relevant eigenvector for later on. In particular, none on the theorems of this section are used elsewhere in the paper; they are simply provided for context. It is worth noting that the limit for $(i^N,j^N,k^N)$ is a linear system, and only comes into force once $y^N$ is close to its equilibrium value $y_*$. \\

Recall that, in terms of original notation, $J^N = II$, $K^N = SI$ and $L^N = SS$.
From the transition rates of the partner model we get the following equations for the drift (see \cite[Section 5]{FED} for a detailed derivation):
\begin{align}
\mu(S^N)  & = 2r_- L^N + r_- K^N  - 2\frac{r_+}{N} \cdot \frac{S^N(S^N-1)}{2} - \frac{r_+}{N} S^NI^N + I^N
\nonumber\\
\mu(I^N) & = 2r_- J^N + r_- K^N - 2\frac{r_+}{N} \cdot \frac{I^N(I^N-1)}{2} - \frac{r_+}{N} S^NI^N - I^N
\nonumber\\
\mu(J^N) & = - r_- J^N  + \frac{r_+}{N} \cdot \frac{I^N(I^N-1)}{2} - 2 J^N + \lambda K^N
\label{ODE1}\\
\mu(K^N) & = - r_- K^N + \frac{r_+}{N}\cdot S^NI^N + 2 J^N  - (\lambda+1) K^N
\nonumber\\
\mu(L^N) & = - r_- L^N + \frac{r_+}{N}\cdot \frac{S^N(S^N-1)}{2} +  K^N
\nonumber
\end{align}
Recalling that $Y^N=S^N+I^N$ is the number of unpartnered individuals and considering only the rate of partnership formation and dissolution, we obtain
\beq
\mu(Y^N)   = r_- (N-Y^N) - \frac{r_+}{N} Y^N(Y^N-1).
\label{Y-drift}
\eeq
Since $y^N_t = Y^N_t/N$,
$$
\mu(y^N) = r_-(1-y^N) - r_+ (y^N)^2 + O(1/N).
$$
Since transition rates are $O(N)$ and $y^N$ jumps by $\pm 2/N$ we have $\sigma^2(y^N) \le CN /N^2 = C/N = o(1)$, for some absolute constant $C$. Using Lemma \ref{lem:limproc} we obtain the following result.

\begin{theorem}\label{thm:y-conv}
 If $y^N_0 \to y_0$ as $N\to\infty$ then $y^N_t \Rightarrow y_t$,
the solution to the initial value problem with $y_0$ and
\beq
\frac{dy}{dt} = r_-(1-y) - r_+ y^2
\label{yode}
\eeq
\end{theorem}
\bigskip

\noindent
From the limiting differential equation, we see that in equilibrium
\beq
0 = r_- (1-y_*) - r_+ y_*^2.
\label{ystar}
\eeq

Since $S^N + I^N + 2(J^N+K^N+L^N) = N$ and $S^N$ is accounted for by $Y^N$, the three remaining equations are those for $I^N,J^N$ and $K^N$. Since $Y^N$ tends towards $Ny_*$ it is helpful to introduce the variable $Z^N=Y^N-Ny_*$ to describe its distance from equilibrium. Most of the terms in these equations are linear, with the exception of terms involving $S_tI_t$ and $I_t(I_t-1)$. Writing $S^N = Y^N-I^N = Ny_* + Z^N-I^N$, from \eqref{ODE1} we find that
\begin{align}
\begin{pmatrix} \mu(I^N) \\ \mu(J^N) \\ \mu(K^N) \end{pmatrix} = A \begin{pmatrix} I^N \\ J^N \\ K^N \end{pmatrix} 
+ r_+ \frac{Z^NI^N}{N}\begin{pmatrix} -1 \\ 0 \\ 1 \end{pmatrix}
+ \frac{r_+(I^N)^2}{N}\begin{pmatrix}0 \\ 1/2 \\ -1 \end{pmatrix}
+ \frac{r_+I^N}{N}\begin{pmatrix}1 \\ -1/2 \\0 \end{pmatrix}
\label{IJK}
\end{align}
where the matrix $A$ is given by
\beq
A = \begin{pmatrix} -(r_+y_*+1) & 2r_- & r_- \\ 0 & -(r_- +2) & \lambda \\ r_+y_* & 2 & -(r_- +\lambda+1) \end{pmatrix}.
\label{Adef}
\eeq
Since $I^N,J^N,K^N$ jump by $O(1)$ at rate $O(I^N+J^N+K^N)$, the diffusivity matrix has entries of size $O(I^N+J^N+K^N)$. 

Results in \cite{Fox} suggest, and our results will show that for any $\ep>0$, after $\ep N^{1/2}$ time $I^N$, $J^N$, and $K^N$ are $O(\sqrt{N})$. Looking at \eqref{IJK}, if $I^N+J^N+K^N$ and $Z^N$ are $o(1)$ then the entries of the diffusivity matrix, as well as the non-linear terms in the drift, are $o(I^N+J^N+K^N)$. Theorem \ref{thm:y-conv} implies that if $Z^N_0=o(N)$ then $Z^N_s=o(N)$ uniformly on (fixed) finite time intervals $[0,t]$. Thus, if we rescale to $i^N_t = I^N_t/\sqrt{N}$, $j^N_t = J^N_t/\sqrt{N}$ and $k^N_t = K^N_t/\sqrt{N}$ (in fact any rescaling by $1/f(N)$ where $f(N)\to\infty$ will work but $f(N)=\sqrt{N}$ is the most relevant to the main results) and use Lemma \ref{lem:limproc} we obtain the following result. 

\begin{theorem}\label{thm:detlim2}
If $y_0^N \to y_*$ and $(i_0^N,j_0^N,k_0^N) \to (i_0,j_0,k_0)$ as $N\to\infty$ then $(i_t^N,j_t^N,k_t^N) \Rightarrow (i_t,j_t,k_t)$, the solution to the initial value problem with $(i_0,j_0,k_0)$ and
\begin{align}
\frac{di}{dt}  & = - (r_+ y_*+1)i  + 2r_- j + r_- k
\nonumber\\
\frac{dj}{dt} & = - r_- j - 2 j + \lambda k
\label{ijk}\\
\frac{dk}{dt} & =  r_+ y_* \cdot i + 2 j  - (r_- +\lambda+1) k
\nonumber
\end{align}
\end{theorem}

\bigskip
Theorem \ref{thm:detlim2} is not used elsewhere in the paper; it only motivates the following calculations.
To analyze the limit behavior of \eqref{ijk}, which is a linear system, we write the condition for an equilibrium as $A (i,j,k)^{\sf T}=0$. To have a nontrivial solution we need $\hbox{det}(A)=0$.
Expanding around the first row
\begin{align}\label{detA}
\hbox{det}(A) & = -(r_+y_*+1) [ (r_-+2)(r_-+\lambda+1) - 2\lambda] \nonumber \\
& - 2 r_-(-\lambda)r_+ y_* + r_-(r_-+2) r_+ y_*
\end{align}
For $\det(A)=0$ we need
\beq
(r_+y_*+1)[ 2 + (3 + \lambda) r_- + r_-^2 ] = r_+ y_* r_- (r_- + 2 + 2\lambda)
\label{detA0}
\eeq
In Section \ref{sec:Rzero} we show that \eqref{detA0} is equivalent to $R_0=1$. This indeed shows that there is a non-trivial equilibrium. We now proceed to find the solution. 
To have $A (i,j,k)^{\sf T} = 0$ we must have
\begin{align}
-(r_+ y_*+1) i + 2r_- j + r_- k & = 0
\nonumber\\
-(r_-+2) j + \lambda k & = 0
\label{ILeq}\\
r_+ y_* i + 2j - (r_- + \lambda + 1) k & = 0
\nonumber
\end{align}
The second equation implies $j = \lambda k/(r_-+2)$. Using this in the first equation we want
$$
-(r_+ y_*+1) i + \frac{2r_- \lambda k}{r_-+2}  + r_- k  = 0.
$$
Solving we see that if
\beq
\alpha = \frac{r_-}{r_+ y_*+1} \left( \frac{2\lambda}{r_-+2} +1 \right) \qquad\text{ and } \qquad \beta = \frac{\lambda}{r_-+2}
\label{eq:invray}
\eeq
then the ray $(\alpha z, \beta z, z)$, $z \ge 0$ is invariant for the dynamical system \eqref{ijk}.

To prove the dynamical system converges to this ray we note that
$$
\theta I - A = \begin{pmatrix} \theta+ (r_+y_*+1) & -2r_- & -r_- \\
0 & \theta+ (r_- +2) & -\lambda \\
-r_+y_* & -2 & \theta + (r_- + \lambda+1) \end{pmatrix}.
$$
The eigenvalues of $A$ are the roots of $0 = \det(\theta I - A) = \theta^3 + b_1 \theta^2
+ b_2 \theta + b_3$, where
$$
b_1 = \hbox{trace}(-A) =  (r_+y_*+1) + (r_- +2) + r_- + (\lambda+1)
$$
and $b_2$ is the sum of the $2 \times 2$ principal minors of $-A$, which is given by
\begin{align*}
&  (r_+y_*+1)(r_-+2) + (r_- +2)(r_- + \lambda+1) - 2\lambda  + (r_+y_*+1)(r_- + \lambda+1) - r_- r_+ y_*.
\end{align*}
Note that the above is positive since each of the two negative terms are cancelled by a part of the positive term that precedes it. Since $b_3 = \det(-A)=0$ and $b_1b_2-b_3>0$ one can use the Routh Hurwicz conditions to conclude that the other two eigenvalues have negative real part. Alternatively one can observe  that the non-zero roots of the equation $\det(\theta I -A)$ are 
$$
\frac{-b_1 \pm \sqrt{b_1^2 - 4 b_2^2}}{2}.
$$
Therefore the dynamical system \eqref{ijk} indeed converges to the invariant ray $(\alpha z, \beta z, z), z \ge 0$. 
{A quantitative statement that applies to the infection process is given in Lemma \ref{ray}.}

\clearp

\section{Worklow and diffusion limits} \label{sec:workflow}

In this section we give a precise statement of the lemmas corresponding to the workflow steps outlined at the end of the Introduction, and use these lemmas to prove the main results, Theorems \ref{finiv},\ref{infiv}, and \ref{T0lim}. The lemmas are listed in the same numerical order as in the Introduction, which is also the order in which they will be proved in subsequent sections.

\subsection{Step 1: the number of singles $Y_t$} \label{subsec:Ybnd}

Recall that (see \eqref{ODE1}-\eqref{Y-drift}) on the original time scale
\begin{align*}
&Y^N \to Y^N+2\quad\hbox{at rate}\quad r_-(N-Y^N)/2 \\
&Y^N \to Y^N-2\quad\hbox{at rate}\quad r_+Y^N(Y^N-1)/2N = r_+(Y^N)^2/2N + O(1).
\end{align*}
If we let $z^N_t = N^{-1/2}(Y^N_t - Ny_*)$, that is, $Y^N_t = Ny_* +N^{1/2}z^N_t$  then
 \begin{align}
 z^N \to z^N + 2/N^{1/2} \quad \hbox{at rate} & \quad r_-[(1-y_*)N - z^NN^{1/2}]/2 \label{eq:z-trans} \\
 z^N \to z^N - 2/N^{1/2} \quad \hbox{at rate} & \quad Nr_+ (y_* + z^NN^{-1/2})^2/2 + O(1) \nonumber \\
 & = N r_+ y_*^2/2 + r_+y_* z^N N^{1/2} + O(1 \vee (z^N)^2). \nonumber
 \end{align}
Since $r_-(1-y_*) = r_+y_*^2$,
\beq
\mu(z^N) = -\mu_z z^N  + O((1 \vee (z^N)^2)/N^{1/2}) \quad\hbox{where}\quad \mu_z = r_-+2r_+y_* \label{mudiff}
\eeq
and
\beq
\sigma^2(z^N) = \sigma^2_z + O(N^{-1} \vee z^N/N^{1/2}) \quad \hbox{where} \quad \sigma^2_z = 2[r_-(1-y_*) + r_+ y_*^2].
\label{sigdiff}
\eeq
So, if $z^N = o(N^{1/4})$, then using Lemma \ref{lem:limproc}, $z_t^N \Rightarrow z_t$ which satisfies
$$
dz = - \mu_z z \, dt + \sigma_z \, dB
$$
i.e., the limit $z_t$ is an Ornstein-Uhlenbeck process. To control the behavior of $z_t^N$ we will prove the following two facts.

\begin{lemma}[Step 1]\label{maxz}
There is a constant $C_{\ref{maxz}}$ so that with high probability,
\begin{itemize}
 \item $|Z^N_t| \le C_{\ref{maxz}}$ for some $t \le C_{\ref{maxz}}\log N$ and
 \item if $|z_0^N| \le (C_{\ref{maxz}}/2)\sqrt{\log N}$ then $|z_t^N| \le C_{\ref{maxz}}\sqrt{\log N}$ for all $t \le N$.
\end{itemize}
\end{lemma}

\vskip10pt

To understand the $\sqrt{\log N}$ scaling note that the stationary distribution of the Ornstein-Uhlenbeck process is a normal,
whose tail scales roughly like $\exp(-x^2/2\sigma^2)/x$; using the heuristic that the time to reach a rare set scales like the reciprocal of its stationary probability and letting $x$ be a constant times $\sqrt{\log N}$, we find the time to reach level $x$ scales roughly like $N^{p}$ for constant $p$, which (with $p=1$) is the time scale on which we control $z^N_t$.

\subsection{Step 2: a special linear combination of $(I^N,J^N,K^N)$}\label{subsec:H}

In Section \ref{sec:O1} we showed that $(i^N,j^N,k^N)$ converges quickly (in $O(1)$ time) to the invariant ray \eqref{eq:invray} for the ODE \eqref{ijk}.
Thus the knowledge of one component determines the other two, provided we have good control on the distance of the triple from the invariant ray. 
Recall that in the example from the Introduction (contact process on a complete graph at criticality), the negative drift of $X_t$ brings it down to the natural spatial scale for the diffusion, $C_\ep N^{1/2}$, within $\ep N^{1/2}$ time. 
This suggests that in our model, we should look for a linear combination of $(I^N,J^N,K^N)$ that has negative drift when it takes values that are $\omega(N^{1/2})$.\\

Motivated by these observations we introduce the variable $H^N=I^N+\gamma J^N+\eta K^N$ where
$(1,\gamma,\eta) A = 0$ and $A$ is given by \eqref{Adef}. Existence is guaranteed since $\det(A)=0$, and the desired constants satisfy
\begin{align}
-(r_+y_*+1)  + r_+ y_* \eta & = 0
\nonumber \\
2r_-  - (r_-+2)\gamma + 2\eta & = 0
\label{leftev}\\
r_-  + \lambda \gamma -(r_-+\lambda + 1)\eta  & =  0
\nonumber
\end{align}
Solving for $\eta$ in the first equation, and $\gamma$ in the second, we find that
\begin{align}\label{eq:eta-gamma}
\eta &= (r_+y_*+1)/r_+y_* \nonumber \\ 
\gamma &= \left( \frac{2r_-}{r_-+2} + \eta \cdot \frac{2}{r_-+2} \right)
\end{align}
Clearly, $\eta>1$, which implies $\gamma>1$. By assumption, $R_0=1$, so in the notation of \cite{FED} we have $\lambda=\lambda_c$. Since $\lambda$ is finite it follows from \cite[Theorem 2.1]{FED} that $r_+y_*>1$ and so $\eta<2$, which easily implies $\gamma<2$. We record these in a display equation for later use:
\begin{align}
1 < \eta < 2 \quad \text{and} \quad 1< \gamma < 2.
\label{eq:eta-gamma2}
\end{align}
We now compute the drift of $H^N$, using \eqref{IJK}. From our choice of linear combination, the linear part drop out, and only the fluctuation term with $Z^N$, the quadratic part, and the lower order term remain:
\beq
\mu(H^N) = (\eta-1) r_+ \frac{Z^NI^N}{N} - (\eta -\gamma/2) r_+ \frac{(I^N)^2}{N} + (1-\gamma/2)r_+\frac{I^N}{N}.
\label{Hdrift}
\eeq
This gets a bit nicer if we rescale in space and time. With $h_N(t) = H^N(\sqrt{N}t)/\sqrt{N}$ we have
\begin{align}\label{fasthdrift}
\mu(h_N) &= (\eta-1) r_+ z_Ni_N - (\eta - \gamma/2) r_+ i^2_N + (1-\gamma/2)i_N/\sqrt{N}
\end{align}
From \eqref{eq:eta-gamma2} we have $\gamma/2 < 1 < \eta$ so the coefficient in the second term is negative. If $i_N = \omega(1)$ the second term dominates the third. Using the bound $|z^N| = O(\sqrt{\log N})$, the second term dominates the first term when $i_N = \omega(\sqrt{\log N})$. Thus, to obtain a closed-form differential inequality for $E[h_N]$ which is useful when $h_N = \omega(\sqrt{\log N})$ it would be enough to show that $I^N/H^N$ is bounded away from zero after a short time, which is done in Section \ref{sec:bdH}. Writing $\tau_{N^{1/5}}^-(H_N) = \inf\{t:H_N(t) \le N^{1/5}\}$ we will prove the following result. Note we are using the slow time scale here.

\begin{lemma}[Step 2] \label{bdI}
Let $\tau = \tau^-_{N^{1/5}}(H_N)\wedge \tau^+_{C_{\ref{maxz}}\sqrt{\log N}}(|z_N|)$. With high probability,
\begin{itemize}
 \item $h_N(t) \le \frac{1}{2}\log N$ for some $t \le 1/\sqrt{\log N}$ and
 \item if $h_N(0) \le \frac{1}{2}\log N$ then $h_N(t) \le \log N$ for all $t \le N^{1/2} \wedge \tau$.
\end{itemize}
\end{lemma}

\vskip10pt

The choice of $N^{1/5}$ as a floor on $H_N$ is so that once $H_N\le N^{1/5}$ a branching process approximation can be used to take $H_N$ to $0$. See Section \ref{sec:T0bd}.

Our reliance on Lemma \ref{maxz} to bound the first part of the drift prevents us from showing $h_N$ comes all the way down to $O(1)$. 
To obtain the stronger bound we will have to show that the first term in the drift in \eqref{fasthdrift} averages out to 0.

\subsection{Step 3: $(I^N,J^N,K^N)$ stays close to the invariant ray}\label{subsec:Q}

To reduce $(I^N,J^N,K^N)$ to a one-dimensional system we let $U^N=I^N/H^N$, $V^N=\gamma J^N/H^N$, and $W^N=\eta K^N/H^N$.
The coefficients in $V^N$ and $W^N$ are there to make $U^N+V^N+W^N=1$.
Recalling the definitions of $\alpha, \beta, \gamma$, and $\eta$ (see \eqref{eq:invray} and \eqref{leftev}) we let $u_* = \alpha/d$, $v_*=\beta \gamma/d$, and $w_* = \eta/d$ where $d=\alpha + \beta \gamma + \eta$
which, as the reader will see, is the fixed point for the dynamical system corresponding to $(U^N,V^N,W^N)$. Let $Q^N= \theta_2(U^N-u_*)^2 + \theta_1(V^N-v_*)^2$. This result is stated on the fast time scale, since this is the time scale on which $Q^N$ naturally converges.

\begin{lemma}[Step 3]\label{ray}
Let $\tau = \tau^-_{N^{1/5}}(H^N) \wedge \tau^+_{C_{\ref{maxz}}\sqrt{\log N}}(|z^N|) \wedge \tau^+_{\log N}(h^N)$. There is a constant $C_{\ref{ray}}$ so that, for any sequence of constants $c_N^Q$ with $N^{-1/6} \le c_N^Q = o(1)$, with high probability,
\begin{itemize}
 \item $\tau^-_{N^{-1/6}/2}(Q^N) \wedge \tau \le C_{\ref{ray}}\log N$, and
 \item if $Q^N_0 \le c_N^Q/2$ then $Q^N_t \le c_N^Q$ for all $t\le N \wedge \tau$.
\end{itemize}
\end{lemma}

Steps 1,2 and 3 could be called the ``a priori'' bounds, since they provide the control needed to implement the averaging result. With this in mind, we make the following definition. Additional constants for both $h$ and $Q$ are specified since we'll need them later.

\begin{definition}\label{def:ctrl-time}
Let $c_N^Q,c_N^h$ be sequences of constants. Say that there is $c_N^h,c_N^Q$ control at time $t$, on the slow time scale, if
\begin{align*}
|z_N(t)| \le C_{\ref{maxz}}\sqrt{\log N}, \ h_N(t) \le c_N^h, \\
Q_N(t) \le c_N^Q \ \text{and} \ H_N(t) > N^{1/5}.
\end{align*}
Define the $c_N^h,c_N^Q$-control time on the slow time scale as
$$\tau_N(c_N^h,c_N^Q,t) = \inf\{s \ge t \ \colon \ \text{there is not} \ c_N^h,c_N^Q \ \text{control at time} \ s \}.$$
Define the control time $\tau_N(\ctl,t)$ as $\tau_N(\ctl,t) = \tau_N(\log N,N^{-1/6},t)$. Define $\tau^N(c_N^h,c_N^Q,t)$ on the fast time scale by $\sqrt{N}\tau_N(c_N^h,c_N^Q,t)$, similarly for $\tau^N(\ctl,t)$. 
\end{definition}

Applying Steps 1-3 (i.e., Lemmas \ref{maxz}, \ref{bdI} and \ref{ray}) sequentially in time, we obtain the following result, which is stated on the slow time scale.

\begin{lemma}\label{notransient}
For each $\ep>0$, with high probability
$$\tau^-_{N^{1/5}}(H_N) \wedge N^{1/2} \le \tau^-_{N^{1/5}}(H_N) \wedge \tau_N(\ctl,\ep).$$
\end{lemma}

In words, this result says that for any fixed $\eps>0$, so long as $H_N$ has not hit the interval $[0,N^{1/5}]$, then w.h.p.\,the variables $|z_N|, Q_N$ and $h_N$ have the desired upper bounds (i.e., those specified by Definition \ref{def:ctrl-time}) on the slow time scale, on the time interval $[\eps,N^{1/2}]$.

\subsection{Step 4: averaging the drift to 0}

Recall from Section \ref{subsec:Ybnd} that $z^N$ is approximately an O.U. process that oscillates on the fast time scale, and once $(i^N,j^N,k^N)$ converges on the invariant ray, we expect it to diffuse along that ray on the slow time scale and thus move slowly when viewed on the fast time scale. Thus, it should not be surprising that we have the following result.

\begin{lemma}[Step 4]\label{aveto0}
Fix $T<\infty$. Let $L:\R^3 \to \R$ be Lipschitz in the $\ell^1$ norm with constant $c_L$ and such that $L(0,0,0)=0$. Let $\omega(N^{-1/2}) = c_N^h \le \log N$ and $0<c_N^Q=o(1)$ be sequences of constants and let $\tau_N = \tau_N(c_N^h,c_N^Q,0)$. With high probability,
$$
\sup_{t \le T} \left|\int_0^{t \wedge \tau_N} z_N(s)L(i_N(s),j_N(s),k_N(s)) \,ds  \right| = O(c_L (c_N^h)^{1/2}(N^{-1/4}\vee (c_N^hc_N^Q)^{1/2})\log^2 N).
$$
\end{lemma}

\noindent

\subsection{Step 5: Show $h_N$ comes down to $O(1)$}

In order to prove Theorem \ref{infiv} we need to show that w.h.p., $h_N(t)=O(1)$ for fixed $t>0$. 

\begin{lemma}[Step 5] \label{downfast}
Let $\tau_C^-(h_N)=\inf\{t \colon h_N(t) \le C\}$. For small enough $\eps>0$,
$$\lim_{C\to\infty} \limsup_N P(\eps/C^3 < \tau_C^-(h_N) \le 1/\eps C) = 1.$$
\end{lemma}

\subsection{Step 6: convergence to diffusion}\label{subsec:weakconv}

Here we prove Theorem \ref{finiv} and Theorem \ref{infiv}. We will need the following result, proved in Section \ref{sec:T0bd}, that shows that once $H_N$ hits $[0,N^{1/5}]$, before long the infection process hits $0$.

\begin{lemma}\label{lem:Hdown}
Suppose that $H^N_0 \le N^{1/5}$, and that w.h.p. $\sup_{t \le N^{1/4}}|z^N_t| \le C_{\ref{maxz}}\sqrt{\log N}$. Then, w.h.p., $\tau_0(H^N) \le N^{1/4}$ and $H^N_t \le N^{.24}$ for all $t \le \tau_0(H^N)$.
\end{lemma}

\mn
\begin{proof}[Proof of Theorem \ref{finiv}]
We use Lemma \ref{lem:limproc} to prove this theorem. 
Since the jump size of $i_N$ is $O(N^{-1/2})=o(1)$ it suffices to find $a,b$ and show convergence of the compensator and predictable quadratic variation. By assumption,
\begin{enumerate}[label={\alph*)}]
\item $|z_N(0)| = O(1)$ so by Lemma \ref{maxz}, for fixed $T>0$, w.h.p.\,$|z_N(t)| \le C_{\ref{maxz}}\sqrt{\log N}$ for all $t \le T$, and
\item $Q_N(0) = O(N^{-\eps})$, so letting $c_N^Q=2Q_N(0)\vee N^{-1/6}$, Lemma \ref{ray} shows that for fixed $T>0$, w.h.p.\,$Q_N(t)\le c_N^Q$ for all $t \le T \wedge \tau^-_{N^{1/5}}(H_N)$.
\end{enumerate}
Point b) implies that $h_N(t) = (1/u_* + O(Q_N(t)^{1/2}))i_N(t) = (1/u_*+o(1))i_N(t)$ for all $t \le T\wedge \tau^-_{N^{1/5}}(H_N)$.
Point a) and Lemma \ref{lem:Hdown} imply that w.h.p.~for all $\tau^-_{N^{1/5}}(H_N)\le t \le T$, $i_N(t) \le h_N(t) = O(N^{0.24-1/2}) = o(1)$. 
Thus if Theorem \ref{finiv} can be proved for $h_N$ for some constants $\mu_*,\sigma_*^2$ then it holds for $i_N$ with different constants. We recall \eqref{fasthdrift}:
\begin{align}\label{eq:fasthdrift2}
\mu(h_N) = (\eta-1) r_+ z_Ni_N - (\eta - \gamma/2) r_+ i^2_N + (1-\gamma/2)i_N/\sqrt{N}
\end{align}
Also $i_N(0),j_N(0),k_N(0) = O(1)$ so the same holds for $h_N(0)$. Letting $c_N^h=\log N$, by Lemma \ref{bdI} and a) above, w.h.p.\,$h_N(t) \le c_N^h$ for all $t \le T \wedge \tau_{N^{1/5}}^-(H_N)$. Combining a) and b) with the bound on $h_N$, w.h.p.
$$\tau_N(c_N^h,c_N^Q,0)\wedge T = \tau_{N^{1/5}}^-(H_N) \wedge T.$$
Using this and Lemma \ref{aveto0}, for any fixed $T>0$, w.h.p.
$$\sup_{t \le T}\left|\int_0^{t \wedge \tau_{N^{1/5}}^-(H_N)} (\eta-1) r_+ z_N(s)i_N(s) ds  \right| = O(N^{-(\frac{\eps}{2} \wedge \frac{1}{4})}(\log^3 N))=o(1),$$
uniformly for $t \le T$. Since $i_N \le h_N$, w.h.p.~for $t \le T \wedge \tau_{N^{1/5}}^-(H_N)$ the third term in the RHS of \eqref{eq:fasthdrift2} is $O(N^{-1/2}\log N)=o(1)$. 
Using the bound on $Q_N$, if $h_N(t) \le R$ for fixed $R>0$ and $t \le T\wedge \tau_{N^{1/5}}^-(H_N)$ then w.h.p.
$$i_N(t) = (u_*+o(1))h_N(t) = u_*h_N(t) + o(1).$$
Looking back to \eqref{eq:fasthdrift2}, we let
$$
b(x) = -\mu_* x^2 \qquad\hbox{with}\quad \mu_* = (\eta-\gamma/2)r_+u_*^2,
$$
and let $\tau^+_R(h_N) = \inf\{t:h_N(t)\ge R\}$. Recall that $h_N^p$ denotes the compensator of $h_N$ (see Section \ref{sec:sampath}). Since $h_N^p(t) = \int_0^t \mu_s(h_N)ds$, we find that w.h.p.
$$\sup_{t \le T \wedge \tau^+_R(h_N) \wedge \tau_{N^{1/5}}^-(H_N)}| h_N^p(t) - \int_0^t b(h_N(s))ds| = o(1).$$
Next let us consider the easier case $\tau^-_{N^{1/5}}(H_N) \le t \le T \wedge \tau^+_R(H_N)$. For this range of values of $t$, from \eqref{eq:fasthdrift2}, and using Lemma \ref{maxz} it easily follows that $\mu(h_N)=o(1)$ and $h_N=o(1)$. So w.h.p.
 $$\sup_{\tau^-_{N^{1/5}}(H^N) \le  t \le T \wedge \tau^+_R(h_N) }| h_N^p(t) - \int_0^t b(h_N(s))ds| = o(1).$$
This proves the assertion about the compensator of $h_N$, as required by Lemma \ref{lem:limproc}. 

Now to calculate the diffusivity of $h$ we let $m$ index the possible transitions and write
$$\sigma^2(h_N) = \sum_m q_m\,(\Delta_m h_N)^2,$$
where $q_m$ is the rate of transition $m$ (it is a function of the state) and $\Delta_m h_N$ is the change in $h_N$ at that transition. 
Note that there are constants $c_m$ so that $(\Delta_m h_N)^2 = c_m/N$. 
Recall that on the fast (original) time scale the transitions of $I^N,J^N,K^N$ have the following rates:
\begin{center}
\begin{tabular}{cc}
transition & rate \\
$I \to S$ & $I^N$ \\
$J \to K$ & $2J^N$ \\
$K \to J$ & $\lambda K^N$ \\
$I+I\to J$ & $r_+(I^N)^2/N$ \\
$I+S \to K$ & $r_+ S^NI^N/N$ \\
$J \to I+I$ &$r_-J^N$ \\
$K \to S+I$ & $r_-K^N$
\end{tabular}
\end{center}

\noindent
Most rates are linear in $I^N,J^N$, or $K^N$. Those which are not are the $I+I\to J$ transition and the $I+S \to K$ transition. As we have seen above, w.h.p., for all $t \le T \wedge \tau^-_{N^{1/5}}(H^N)$, $I_N(t) \le N^{1/2}\log N$.  Therefore the $I+I\to J$ transition has rate $O((\log N)^2)$. By a similar reasoning the $I+S \to K$ transition has rate $r_+(y_* + O(N^{-1/2}\log N))I^N = r_+y_*I^N + O((\log N)^2)$.
Speeding up time by $N^{1/2}$ and writing in lower case, the rates are equal to $Ni_N$, $2Nj_N$, etc and the error terms have rate $O(\sqrt{N}(\log N)^2)$. Since each of $i_N(t),j_N(t),k_N(t)$ is equal to $(\text{constant} + o(1))h_N(t)$ for $t \le T \wedge \tau_{N^{1/5}}^-(H_N)$, if in addition $h_N(t) \le R$ then there are constants $d_m$ so that for each $m$,
$$q_m = Nd_m h_N(1+ o(1)) + O(\sqrt{N}(\log N)^2)) = Nd_mh_N + o(N).$$
Thus there is a constant $\sigma_*^2$ so that if $h_N \le R$ and $t \le T \wedge \tau^-_{N^{1/5}}(H_N) \wedge \tau^+_R(h_N)$ then w.h.p.
$$\sigma^2(h_N) = \sum_m Nd_mh_N\frac{c_m}{N} + o(1) = \sigma_*^2 h_N + o(1).$$
If $\tau^-_{N^{1/5}}(H_N) \le t \le T$ then since w.h.p.~$h_N(t)=O(N^{.24-1/2}) = o(1)$, an easy computation gives $\sigma^2(h_N)=o(1)$. 
This implies the desired convergence of predictable quadratic variation with $a(x) = \sigma_*^2 x$. An application of Lemma \ref{lem:limproc} now shows $h_N$ (and hence $i_N$) converges to a diffusion of the desired form.
\end{proof}

\vskip10pt
We now prove Theorem \ref{infiv}.

\vskip10pt

\begin{proof}[Proof of Theorem \ref{infiv}]
We first explain why it makes sense to start the limiting diffusion $X$ from $\infty$. For $C>0$, let $\tau^-_C(X) = \inf\{t \colon X_t \le C\}$. Using Jensen's inequality,
$$\frac{d}{dt}E[X_t] = -\mu E[X_t^2] \le -\mu (E[X_t])^2,$$
so we find that
$$E[X_t \mid X_0=x] \le \frac{1}{\mu t + 1/x}$$
and using Markov's inequality,
$$P(\tau^-_C(X) \ge t \mid X_0=x) \le C^{-1} E[X_t \mid X_0=x] \le \frac{1}{C\mu t  + C/x}.$$
In particular, if $x,C\to\infty$ with $C\le x$ then $\tau^-_C(X)$ converges in probability to zero. It is then not hard to show that the law of $X$, conditioned on $X_0=x$, converges in distribution as $x\to\infty$.\\

Since the limiting diffusion is continuous, it crosses any level $C>0$, if it starts from $\infty$. Thus if we let $\tau_C(X) = \inf\{t:X_t=C\}$, $\tau_C(X) \da 0$ in probability as $C \ua \infty$. 
Since convergence in distribution allows for small time change, given the proof of Theorem \ref{finiv}, it is enough to show there are sequences $\ep_m\to 0, C_m \to \infty$ so that for each $m$, w.h.p.~there is a $t_m \le \ep_m$ such that
\begin{itemize}[noitemsep]
 \item $h_N(t_m)=C_m + O(N^{-1/2})$,
 \item $|z_N(t_m)| \le C_{\ref{maxz}}\sqrt{\log N}$ and
 \item $|Q_N(t_m)| \le N^{-1/6}$.
\end{itemize}
Letting $C_m = m$ and $\ep_m = 1/\eps m$ for small $\eps>0$, Lemma \ref{downfast} gives the bound on $h_N$. Since Lemma \ref{downfast} also gives $t_m \ge \eps/m^3$, we may apply Lemma \ref{notransient} to obtain the desired bounds on $|z_N|$ and $Q_N$. This completes the proof.
\end{proof}

\subsection{Step 7: extinction time}\label{subsec:extime}

To prove a result for the time for the infection to die out, note that $\tau_0(h_N) = \inf\{t \colon h_N(t)=0\}$ is the first time (on the $N^{1/2}$ time scale) there are no infected individuals.\\

The continuous mapping theorem makes half of the
proof easy. To complete it we need to show that $\tau_0(h_N)$ converges in probability to $0$ as $h_N(0) \to 0$, uniformly for large $N$. This is accomplished by combining Lemma \ref{lem:Hdown} with the following result, that by the definition of $\tau(\ctl,0)$ implies that if $h_N$ is initially small and if the values of $|z_N|,Q_N$ can be kept under control then within a short time, $H_N$ hits $[0,N^{1/5}]$.

\begin{lemma}\label{lem:Hdn1}
Fix $\ep>0$. There is $\dlt>0$ so that
$$P(\tau_N(\ctl,0) < \tau^+_{\log N}(h_N) \wedge \eps \mid h_N(0) \le \dlt) \ge 1-\eps.$$
\end{lemma}

\vskip10pt

Lemmas \ref{lem:Hdown} and \ref{lem:Hdn1} are postponed to Section \ref{sec:T0bd}. For now we use them to prove Theorem \ref{T0lim}.

\begin{proof}[Proof of Theorem \ref{T0lim}]
 
Recall that $\mathcal{Q}$ is the law of the limiting diffusion for $h_N$ and $\tau_x(X)$ is the hitting time of $x>0$ for the limiting diffusion. By the strong Markov property and Blumenthal's 0-1 law, after hitting $x$, the process
will with probability one immediately hit $(0,x)$ and $(x,\infty)$. 
From this it follows easily that $\tau_x(X) : C \to \RR$ is continuous $\mathcal{Q}$-almost surely. 
Suppose the infection process satisfies the conditions of Theorem \ref{finiv} with $h_N(0) \to x$. 
Let $P_N$ denote its law and let $\tau^-_x(h_N)  = \inf\{t:h_N(t) \le x\}$. Using the continuous mapping theorem,
$$
P_N( \tau^-_x(h_N) > t ) \to \mathcal{Q}( \tau_x(X) > t)
$$
Let $\ep>0$. If $x$ is small enough then $\mathcal{Q}_x(\tau_0(X) > \ep) < \ep$ and hence $\mathcal{Q}(\tau_0(X) > t+\ep) \le \mathcal{Q}(\tau_x(X) >t ) + \ep$. Combining this with the last result
and noting that $x \mapsto \tau^-_x(h_N)$ is increasing,
$$
\liminf_{N \to \infty} P_N( \tau_0(h_N) > t ) \ge \mathcal{Q}( \tau_0(X) > t+\ep) - \ep.
$$
Letting $\ep\to 0$ we have half of the desired convergence in distribution. 

To get the other half we again fix any arbitrary $\ep >0$. Note that for $\dlt>0$
$$
P_N( \tau_0(h_N) > t + \ep) \le P_N( \tau^-_{\dlt}(h_N)>t ) + \sup_{h \le \dlt}P_N(\tau_0(h_N) > \ep \mid h_N(0) = h).
$$
Letting $N\to\infty$, $P_N( \tau^-_\dlt(h_N)>t) \to \mathcal{Q}( \tau_\dlt(X) > t) \le \mathcal{Q}( \tau_0(X) > t)$. So it suffices to show that for each $\ep>0$ there is a $\dlt>0$ so that the second term is at most $O(\ep)$, uniformly for large $N$. Since by convergence in distribution, on the $N^{1/2}$ time scale it takes at least $s>0$ amount of time, w.h.p.\,as $s\to 0^+$ and $N\to\infty$, for $h$ to reach $\dlt$ if $h_N(0) \to x >\dlt$, using Lemma \ref{notransient} we may assume when taking the above $\sup$ that not only $h\le \dlt$ but also that $|z_N(t)|$ and $Q_N(t)$ are small (in the sense of the definition of $\tau(\ctl,t)$ given in Definition \ref{def:ctrl-time}) for all $t \le \omega(1) \wedge \tau_{N^{1/5}}^-(H_N)$. The desired bound then follows directly from Lemmas \ref{lem:Hdown} and \ref{lem:Hdn1}. This completes the proof.
\end{proof}

\clearp

\section{Step 1: upper bound on $|z^N|$} \label{sec:bdY}

Let $C_{\ref{maxz}}$ be a sufficiently large constant. In this short section we prove Lemma \ref{maxz}, in two parts:
\begin{itemize}[noitemsep]
\item \textit{approach:} show that $|Z^N_t| \le C_{\ref{maxz}}$ for some $t \le C_{\ref{maxz}}\log N$ w.h.p., then
\item \textit{control:} show that if $|z^N_t| \le (C_{\ref{maxz}}/2)\sqrt{\log N}$ then w.h.p., $|z^N_t| \le C_{\ref{maxz}}\sqrt{\log N}$ for all $t \le N$.
\end{itemize}

\vskip10pt

\noindent\textbf{Approach.} Recall from \eqref{Y-drift} that
$$\mu(Y^N) = F_N(Y^N) \quad \hbox{where} \quad F_N(Y^N) = r_-(N-Y^N) - \frac{r_+}{N}Y^N(Y^N-1),$$
where, for each $N$, $F_N:\RR \to \RR$ is just some function. Let $Y_*^N \in (0,N)$ be the unique value with $F_N(Y_*^N)=0$.
Note that $Y_*^N \ne Ny_*$ since we used $Y^N(Y^N-1)$ and not $(Y^N)^2$ to compute it. However note that $F_N(Ny_*) = r_+y_*$ and by concavity of $F_N$, $|F_N'|$ is bounded below by $|F_N'(0)| = r_- - r_+/N$, so
$$
|Y_*^N - Ny_*| \le r_+y^*/(r_--r_+/N) = O(1),
$$
and letting $\tilde Z^N = Y^N-Y_*^N$ we have $\tilde Z^N - Z^N = Y_*^N-Ny_* = O(1)$, so it is enough to prove the result with $\tilde Z^N$ in place of $Z^N$. Since $F_N$ is concave, if we let $r_N = F_N(0)/Y_*^N$ then for $Y^N \in [0,N]$ with $Y^N \ne Y_*^N$ we have
$$F_N(Y^N)/(Y^N-Y_*^N) \le -r_N.$$
Letting $r=\frac{1}{2}\liminf_{N\to\infty} r_N$, it follows that $F_N(Y_*^N + \tilde Z^N)/\tld Z^N \le -r$ for large $N$ and $\tilde Z^N \ne 0$. 
If $\tilde Z^N_0 \ge 1$ then letting $\tau_2^-(\tilde{Z}^N) = \inf\{t: \tilde Z^N_t \le 2\}$ and using the product rule (Lemma \ref{lem:qac-prod}), for $t<\tau_2^-(\tilde Z^N)$
$$\mu(e^{rt}\tilde Z^N_t) = e^{rt}(r\tilde Z^N_t + F(Y_* + \tilde Z^N_t)) \le 0,$$
so $\xi_t = \exp(r(t \wedge \tau_2^-(\tilde Z^N)))\tld Z^N_{t\wedge \tau_2(\tilde Z^N)}$ is a supermartingale.
If $\tau_2^-(\tilde Z^N)>t$ then $\xi_t \ge 2e^{rt}$. Moreover $\xi_0 =\tilde Z^N_0 \le N$. So
$$P(\tau_2^-(\tilde Z^N) > t) \le P(\xi_t \ge 2e^{rt}) \le \frac{e^{-rt}}{2} E[\xi_t] \le e^{-rt}N/2.$$
If $\tilde Z^N_0 \le -1$ then letting $\tau_2^-(-\tilde Z^N) = \inf\{t: \tilde Z^N_t \ge -2\}$, we obtain the same estimate for $P(\tau_2^-(-\tilde Z^N) >t)$, so for $\tau_2^-(|\tilde Z^N|) = \inf\{t:|Z^N_t| \le 2\}$, taking a union bound and $t=C\log N$ with $C = 2/r$ we find
$$P(\tau_2^-(\tilde Z^N)>(2/r)\log N) \le 1/N = o(1),$$
which proves the result.

\vskip10pt

\noindent\textbf{Control.} We use Lemma \ref{lem:driftbar} to control $z^N = N^{-1/2}Z^N$. 

Let $x = (C_{\ref{maxz}}/2)\sqrt{\log N}$ with $C_{\ref{maxz}}$ to be determined and let $X = -x + |z^N-x|$, then $\Delta_\star(X) \le 2/N^{1/2}\le x/2$ for large $N$. Since $z^N$ jumps by at most $x$, if $|z^N|\ge x$ then $\mu(X) = \sgn(z^N)\mu(z^N)$. From the proof of approach, $\mu(z^N)/z^N \le - r$ if $z^N \ne 0$. If $X \ge x/2$ then $|z^N| \ge 3x/2$ so letting $\mu_\star=3rx/2$, $\mu(X) \le -\mu_\star$.  If $X\le x$ then $|z^N| \le 2x$ so using \eqref{mudiff}-\eqref{sigdiff}, $|\mu(z^N)|\le Cx$ and $\sigma^2(z^N)\le C$ for some $C>0$, so let $C_{\mu_\star}$ be a large enough multiple of $x$ and $\sigma^2_\star$ a large enough constant. Then, $\Delta_\star(X)\mu_\star/\sigma^2_\star = o(1)$ which allows us to take $C_\Delta=1$. Then, $\Gamma \ge \exp(\delta L^2)$ and $\lfloor \Gamma \rfloor x/16C_{\mu_\star} \ge \dlt \Gamma$ for some $\dlt>0$. Taking $C_{\ref{maxz}} > 1/\sqrt{\dlt}$ makes $\Gamma,\dlt\Gamma \ge N$ for large $N$ and the result follows from Lemma \ref{lem:driftbar}.

\section{Step 2: upper bound on $h_N$} \label{sec:bdH}

Recall $H^N = I^N+\gamma J^N + \eta K^N$ and $\tau^-_{N^{1/5}}(H^N)$ is the first time that $H^N \le N^{1/5}$. 
Recalling \eqref{Hdrift} we see that the negative term in $\mu(H^N)$ involves $I^N$. Therefore, the first step is to get a lower bound on $I^N/H^N$. Note the event below is taken to be vacuous if $\tau^-_{N^{1/5}}(H^N)<C_{\ref{bdHI}}$.

\begin{lemma} \label{bdHI}
There is a constant $C_{\ref{bdHI}}>0$ so that with high probability,
$$H^N_t \le C_{\ref{bdHI}} I^N_t \quad \hbox{for all} \quad C_{\ref{bdHI}} \le t \le N \wedge \tau^-_{N^{1/5}}(H^N).$$
\end{lemma}

\begin{proof}
Let $\ep>0$ be a small enough constant, then there are two steps:
\begin{itemize}[noitemsep]
\item \textit{approach}: show that $\tau^+_\ep(I^N/H^N) \wedge \tau^-_{N^{1/5}}(H^N) \le 1/\ep$ w.h.p., then
\item \textit{control}: show that if $I^N_0/H^N_0\ge \ep$ then w.h.p.\,$I^N_t/H^N_t \ge \ep/2$ for all $t \le N\wedge \tau^-_{N^{1/5}}(H^N)$.
\end{itemize}

\noindent\textbf{Approach.} From equations \eqref{IJK} and \eqref{Hdrift}, we have
\begin{align}\label{eq:IH}
\mu(I^N) &=  - r_+ \frac{Z^N}{N}I^N - (r_+y_* + 1)I^N + 2r_- J^N + r_- K^N + r_+\frac{I^N}{N}  \\
\mu(H^N) &= (\eta-1) r_+ \frac{Z^N}{N}I^N - (\eta -\gamma/2) r_+ \frac{(I^N)^2}{N} + (1-\gamma/2)r_+\frac{I^N}{N}. \nonumber
\end{align}
Using the product rule of Lemma \ref{lem:qac-prod} and the Taylor approximation of Lemma \ref{lem:taylor}, we find that if $H^N>N^{1/5}$, then since $H^N=\omega(1)$,
\begin{align}
\mu\left (\frac{I^N}{H^N} \right) &= \frac{\mu(I^N)}{H^N} + I^N\mu\left(\frac{1}{H^N}\right) + \sigma\left (I^N,\frac{1}{H^N} \right) \label{IoH}\\ 
&= \frac{\mu(I^N)}{H^N} - \frac{I^N}{(H^N)^2} \mu(H^N) + \sigma^2(H^N) \, O\bigg(\frac{I^N}{(H^N)^3}\bigg)+ \sigma\left (I^N,\frac{1}{H^N} \right).  \nonumber
\end{align}
The assumption $H^N=\omega(1)$ is used together with the fact $\Delta_\star(H^N)=O(1)$ to obtain $O(I^N/(H^N)^3)$ from the Taylor approximation. Next we show the last two terms are $O(1/H^N)$. To compute $\sigma(I^N,1/H^N)$ note that $I^N$ or $H^N$ jumps at rate $O(H^N)$, $I^N$ jumps by $O(1)$, and $1/H^N$ jumps by $O(1/(H^N)^2)$ when $H^N=\omega(1)$. Multiplying, we obtain $\sigma(I,1/H^N) = O(1/H^N)$. By a similar argument $\sigma^2(H^N)=O(H^N)$, and since $I^N\le H^N$ it follows that $\sigma^2(H^N)I^N/(H^N)^3 = O(1/H^N)$.\\

If $I^N/H^N \le \ep$ then
$$\max\{ \gamma J^N/H^N, \eta K^N/H^N \} \ge (1-\ep)/2.$$
Since $|Z^N|,I^N \le N$ it follows that $\mu(H^N) = O(I^N)$. 
From \eqref{eq:IH} and \eqref{IoH} we then deduce that if $I^N/H^N \le \ep$ and $H^N = \omega(1)$ then
\begin{equation}\label{eq:mu_lbd}
\begin{array}{rl}
\mu(I^N/H^N) & \ge -\ep(r_+ + r_+y_*+1) + r_-\min(2/\gamma,1/\eta)(1-\ep)/2 \\
 & - O(\ep^2) - o(1).
 \end{array}
 \end{equation}
If $\ep>0$ is taken small enough, the right-hand side is at least a constant $\mu_0>0$. To estimate $\sigma^2(I^N/H^N)$ we note that if $H^N=\omega(1)$ then $I^N/H^N$ jumps at rate $O(H^N)$, by an amount $O(1/H^N)$ when $I^N$ jumps, and an amount $O(I^N/(H^N)^2) = O(1/H^N)$ (since $I^N\le H^N$) when $H^N$ jumps. Thus $\sigma^2(I^N/H^N) = O(1/H^N)$ when $H^N=\omega(1)$.\\

To summarize our progress so far, if we let
$$\tau_1 = \inf\{t\colon I^N_t/H^N_t > \ep \ \text{or} \ H^N_t \le N^{1/5}\}$$
and note that $t<\tau_1$ implies $H^N>N^{1/5}=\omega(1)$ and $1/H^N < N^{-1/5}$, then there are constants $\mu_0,C>0$ such that
\begin{align}\label{eq:hi-summ}
\mu_t(I^N/H^N) \ge \mu_0\quad\text{and}\quad\sigma^2(I^N/H^N) \le CN^{-1/5} \ \text{for all} \ t<\tau_1.
\end{align}
Define $\xi_t=I^N_{t \wedge \tau_1}/H^N_{t\wedge \tau_1}$. In the notation of Section \ref{sec:sampath},
$$\xi_t^m = \xi_t - \xi_0 - \xi_t^p \le \xi_t - \mu_0 (t\wedge \tau_1)\quad\text{and}\quad\lng \xi \rng_t \le CN^{-1/5}t \wedge \tau_1.$$
So for $a,\phi >0$
\begin{equation}\label{eq:ineq-compen}
\xi_t^m + (a + \phi \, \lng \xi \rng_t ) \le \xi_t + a - (\mu_0 - \phi CN^{-1/5})t \wedge \tau_1.
\end{equation}
Using Lemma \ref{lem:sm-est}, if $\phi\Delta_\star(\xi) \le 1/2$, then the left-hand side of \eqref{eq:ineq-compen} is $\ge 0$ for all $t\ge 0$ with probability $\ge 1-2e^{-\phi a}$, or in other words,
$$P(\xi_t \ge -a + (\mu_0-\phi CN^{-1/5})t \wedge \tau_1 \ \text{for all} \ t\ge 0) \ge 1-2e^{-\phi a}.$$
Letting $a=\ep$ and $\phi = (\mu_0/2C)N^{1/5}$ gives $e^{-\phi a} = o(1)$. Since $\xi$ is stopped if ever $H^N \le N^{1/5}$, it follows that $\Delta_\star(\xi)$, which is $O(1/H^N)$, is a fortiori $O(N^{-1/5})$, which means that $\phi\Delta_\star(\xi) \le 1/2$ if $C>0$ is taken large enough. Summarizing, we find that
$$P(\xi_t \ge -\ep + \mu_0 t \wedge \tau_1/2 \ \text{for all} \ t \ge 0) = 1-o(1).$$
Since $t>\tau_1$ if ever $\xi_t\ge \ep$, it follows that
$$P(\tau_1 > 4\ep/\mu_0) = 1-o(1).$$
Since $4\ep/\mu_0$ is a constant and $\tau_1 = \tau^-_{N^{1/5}}(H^N) \wedge \tau^+_\ep(I^N/H^N)$, the first part is proved.\\

\noindent \textbf{Control.} We now show that if $I^N_0/H^N_0 \ge \ep$ then w.h.p.\,$I^N_t/H^N_t \ge \ep/2$ for all $t\le N \wedge \tau^-_{N^{1/5}}(H^N)$. To do so we use Corollary \ref{cor:driftbar} to Lemma \ref{lem:driftbar}. Let $\tau = \smash{\tau^-_{N^{1/5}}(H^N)}$, let $X_t = \ep-I^N(t\wedge \tau)/H^N(t\wedge \tau)$,
and let $x = \ep/2$. Similarly as for $\xi$, we find that $\Delta_\star(X) = O(1/H^N) = O(N^{-1/5})$, which is $o(x)$. From \eqref{eq:IH}-\eqref{IoH} it is easy to check that $|\mu(I^N/H^N)| = O(1)$ when $H^N=\omega(1)$ so let $C_{\mu_\star}$ be a large constant. From \eqref{eq:hi-summ} we have $\mu_t(X) \le -\mu_0$ and $\sigma^2_t(X)\le CN^{-1/5}$ when $t<\tau$ and $X_t\ge 0$, so let $\mu_\star=\mu_0$ and $\sigma^2_\star = CN^{-1/5}$ for some $C>0$. In this way $\Delta_\star(X)\mu_\star/\sigma^2_\star = O(1)$ which allows us to let $C_\Delta$ be a large constant. Then, $\Gamma \ge \exp(\dlt N^{1/5})$ and $\lfloor\Gamma \rfloor x/16C_{\mu_\star} \ge \dlt\Gamma \ge N$ for some $\dlt>0$ and large $N$, so Corollary \ref{cor:driftbar} gives
$$\lim_{N\to\infty}P\left( \sup_{t < N\wedge \tau^-_{N^{1/5}}(H^N)} I^N_t/H^N_t \le \ep/2 \ \mid \ I^N_0/H^N_0 \ge 3\ep/4 \right) = 0,$$
which completes the proof.
\end{proof}

\vskip10pt

\begin{proof}[Proof of Lemma \ref{bdI}] The result has two parts.\\

\noindent\textbf{Approach.} First we show that
\begin{center}
w.h.p., $h_N(t) \le \frac{1}{2}\log N$ for some $t \le 1/\sqrt{\log N}$.
\end{center}
Note the slow time scale is used. We will need the estimates we've proved so far. Let
\begin{align*}
\tau_1 &= \inf\{t \ge C_{\ref{maxz}}N^{-1/2}\log N \colon |z_N(t)| > C_{\ref{maxz}}\sqrt{\log N}\}  \quad \text{and} \quad \\
\tau_2 &= \inf\{t \ge C_{\ref{bdHI}}N^{-1/2}\colon h_N(t) > C_{\ref{bdHI}}i_N(t) \ \text{or} \ H_N(t)\le N^{1/5}\}.
\end{align*}
By Lemma \ref{maxz}, $\tau_1>N^{1/2}$ w.h.p., and by Lemma \ref{bdHI}, $\smash{\tau_2 > N^{1/2} \wedge \tau^-_{N_{1/5}}(H_N)}$ w.h.p. It is more convenient if we shift the time variable over by $-C_{\ref{maxz}}N^{-1/2}\log N$ so that both estimates begin to hold at $t=0$; since $N^{-1/2}\log N=o(1/\sqrt{\log N})$ this will not affect the conclusion.

We recall \eqref{fasthdrift}:
$$
\mu(h_N) = (\eta-1) r_+ z_Ni_N - (\eta - \gamma/2) r_+ i^2_N + (1-\gamma/2)i_N/\sqrt{N}
$$
From \eqref{eq:eta-gamma2}, $\eta > 1 > \gamma/2$, moreover $i_N \le h_N$, so for $t<\tau_1 \wedge \tau_2$,
$$
\mu(h_N) \le \left((\eta-1) r_+ C_{\ref{maxz}}\sqrt{\log N} + (1-\gamma/2)r_+\frac{1}{\sqrt{N}}\right)\, \frac{h_N}{C_{\ref{bdHI}}} - (\eta-\gamma/2) r_+ \frac{(h_N)^2}{C^2_{\ref{bdHI}}}.
$$
If $h_N>\frac{1}{2}\log N$ the first term is $o((h_N)^2)$. Since $\eta>\gamma/2$, we see there is $\delta>0$ so that, if we let $\tau_3=\tau_1\wedge \tau_2 \wedge \tau^-_{\frac{1}{2}\log N}(h_N)$, then for $t<\tau_3$,
$$\mu(h_N) \le -\delta \, (h_N)^2.
$$
Next we'd like to set up a differential inequality for $h_N$; the only trouble is, the drift estimate only holds up to a stopping time. To fix this, let $\psi(t;h)$ denote the solution flow for the differential equation $h'(t)=-\delta h(t)^2$ (i.e., the function with maximal domain containing $\{0\}\times \R$ such that $\partial_t\psi(t;h) = -\delta \psi(t;h)^2$ and $\psi(0;h)=h$) and define the continued process
$$\tilde h_N(t) = \psi(t-t\wedge \tau_3;h_N(t \wedge \tau_3)).$$
In words, $\tilde h_N$ is equal to $h_N$ up to time $\tau_3$, at which point it evolves according to the flow $\psi$. It is then clear that for all $t \ge 0$,
$$\mu(\tilde h_N) \le -\delta \, (\tilde h_N)^2.$$
Taking expectations and using Jensen's inequality we then find
$$\frac{d}{dt}E[\tilde h_N(t)] \le -\delta \left(E[\tilde h_N(t)] \right)^2,$$
which implies $E[\tilde h_N(t)] \le \psi(t;E[h_N(0)])$ (note $\tilde h_N(0)=h_N(0)$). Solving the DE we have
$$\psi(t;h) = 1/(\delta t + 1/h) \le 1/(\delta t).$$
If $\tau_3>t$ then $\tilde h_N(t) = h_N(t) > \frac{1}{2}\log N$, so it follows that
$$P(\tau_3>t) \le P(\tilde h_N(t) = h_N(t) > \frac{1}{2}\log N) \le \frac{2}{\log N}E[\tilde h_N(t)] \le 2(\delta t \log N)^{-1}.$$
Taking $t = 1/\sqrt{\log N}$ the above is $o(1)$. As noted above, $\tau_1\wedge \tau_2 > N^{1/2} \wedge \tau^-_{N^{1/5}}(H_N)$ w.h.p. Since $h_N>\frac{1}{2}\log N$ implies $H_N>N^{1/5}$, if $h_N(s)>\frac{1}{2}\log N$ for all $s\le t$ then w.h.p.\, $\tau_3>t$. It follows that
$$P(h_N(s)>\frac{1}{2}\log N \ \text{for all} \ s \le 1/\sqrt{\log N}) = P(\tau_3> 1/\sqrt{\log N}) + o(1) = o(1)$$ 
and the statement is proved.\\

\noindent\textbf{Control.} We now show that
\begin{center}
if $h_N(0) \le \frac{1}{2}\log N$ then w.h.p.\,$h_N(t) \le \log N$ for all $t \le N^{1/2} \wedge \tau^-_{N^{1/5}}(H_N) \wedge \tau^+_{C_{\ref{maxz}}}(|z^N|).$
\end{center}
This time, let $\smash{\tau_1 = \tau^+_{C_{\ref{maxz}}}(|z_N|)}$, and define $\tau_2$ as in the proof of approach. By Lemma \ref{bdHI},\\ $\smash{\tau_2 > N^{1/2} \wedge \tau^-_{N_{1/5}}(H_N)}$. For the present result we cannot ignore small times, so first we show the conclusion holds for all $t\le C_{\ref{bdHI}}N^{-1/2}$ without assuming a bound on $i_N/h_N$. Let $\tau_3 = \tau_1 \wedge \tau^+_{\log N}(h_N)$. From \eqref{fasthdrift}, if $t<\tau_3$ then since $i_N \le h_N$, $\eta-\gamma/2>0$ and $z_N(t)=O(\sqrt{\log N})$, for large $N$ we have $\mu(h_N(t)) \le \log(N)h_N(t)$. Since $h_N$ jumps by $O(N^{-1/2})$ at rate $Nh_N$ on the slow time scale, $\sigma^2(h_N)\le Ch_N$ for some constant $C>0$. Let $\xi_t = h_N(t \wedge \tau_3)$, then for $a,\phi>0$ and all $t\ge 0$,
$$\xi_t^m - a -\phi \lng \xi \rng_t \ge \xi_t - \xi_0 - a - \big((\log N)^2+C\phi \log N)t.$$
Using Lemma \ref{lem:sm-est}, if $\phi\Delta_\star(\xi)\le 1/2$, the LHS is $\le 0$ for all $t \ge 0$ with probability at least $1-2e^{-\phi a}$. Letting $\phi = (\log N)/C$ and $a=1$, $\phi\Delta_\star(\xi) = O(N^{-1/2}\log N)\le 1/2$ for large enough $N$ and $e^{-\phi a}=o(1)$, so
$$P(\xi_t \le \xi_0 + 1 + 2t(\log N)^2 \ \text{for all} \ t \ge 0) = 1-o(1).$$
It follows that $P(\xi_t \le \log N \ \text{for all} \ t\le C_{\ref{bdHI}}N^{-1/2})=1-o(1)$, since $\xi_0 \le \frac{1}{2}\log N$ and $1 + 2t(\log N)^2 \le \frac{1}{4}\log N$ for all $t \le C_{\ref{bdHI}}N^{-1/2}$, which implies that
$$P(h_N(t) \le \frac{3}{4}\log N \ \text{for all} \ t \le C_{\ref{bdHI}}N^{-1/2} \wedge \tau_1) = 1-o(1).$$
With the above estimate in hand, to prove the desired result we may assume $h_N(0) \le \frac{3}{4}\log N$, redefine $\tau_2 = \tau^+_{C_{\ref{bdHI}}}(h_N/i_N) \wedge \tau^-_{N^{1/5}}(H_N)$ and suppose that $\tau_2 > N^{1/2} \wedge \tau^-_{N^{1/5}}(H_N)$ w.h.p. We will use Corollary \ref{cor:driftbar}. Let $\tau = \tau_1 \wedge \tau_2$, let $X_t= h_N(t\wedge \tau) - \log(N)/2$ and let $x = \log(N)/2$, then $\Delta_\star(X) = O(N^{-1/2}) = o(x)$. Since $h_N$ jumps by $O(N^{-1/2})$ at rate $O(Nh_N)$, if $h_N \le \log N$ then $|\mu(h_N)|=O(N^{1/2}h_N) = O(N^{1/2}\log N)$ and $\sigma^2(h_N) = O(h_N) = O(\log N)$, so let $C_{\mu_\star}= CN^{1/2}\log N$ and $\sigma^2_\star= C\log N$ for large $C$. From the proof of approach, if $t<\tau$ and $h_N(t) \ge \log(N)/2$ then $\mu_t(h_N) \le -\delta (h_N)^2$, so let $\mu_\star = -\delta  (\log N)^2/4$. With these choices $\Delta_\star(X)\mu_\star/\sigma^2_\star = o(1)$ so let $C_\Delta=1$. We now find $\mu_\star x/\sigma^2_\star = \Omega((\log N)^2)$. Therefore, $\Gamma = e^{\Omega((\log N)^2)} = \omega(N)$ and $x/16C_{\mu_\star} = \Omega(N^{-1/2})$. Using Corollary \ref{cor:driftbar},
$$P( h_N(t) \ge \log N \ \  \hbox{for some} \ \ t \le N^{1/2} \wedge \tau \mid h_N(t) \le \frac{3}{4}\log N) = o(1),$$
and the result follows since $N^{1/2} \wedge \tau = N^{1/2} \wedge \tau^-_{N^{1/5}}(H_N) \wedge \tau_1$ w.h.p.
\end{proof}

\clearp

\section{Step 3: $(I^N,J^N,K^N)$ stays close to the invariant ray} \label{sec:ray}

In this section we prove Lemma \ref{ray}. As in the statement of the lemma, let
$$\tau = \tau^-_{N^{1/5}}(H^N) \wedge \tau^+_{C_{\ref{maxz}}\sqrt{\log N}}(|z^N|) \wedge \tau^+_{\log N}(h^N).$$
Looking back to \eqref{IoH}, we showed that if $H^N=\omega(1)$ then
$$\sigma(I^N,1/H^N)+\sigma^2(H_N) O(I^N/(H^N)^3)=O(1/H^N);$$
it is clear the same estimate holds with $J^N,K^N$ in place of $I^N$. If, moreover, $H^N,|Z^N|\le N^{1/2}\log N$, then since $\eta,\gamma\ge 1$ (see \eqref{eq:eta-gamma2}), $I^N,J^N,K^N \le H^N$, from \eqref{IoH} we obtain the estimate $\mu(H_N) = O((\log N)^2)$. Recalling that $(U^N,V^N,W^N) = (I^N,\gamma J^N,\eta K^N)/H^N$, writing \eqref{IoH} but with $\gamma J^N$ and $\eta K^N$ as well, and using the above estimates, we find that if $t<\tau$ then
\begin{align}
\mu_t\begin{pmatrix} U^N \\ V^N \\ W^N \end{pmatrix} = 
\frac{1}{H^N_t} \ \mu_t \begin{pmatrix} I^N \\ \gamma J^N \\ \eta K^N \end{pmatrix} + O((\log N)^2N^{-1/5}). \label{muUVW}
\end{align}
Let $\mathcal{J}^N$ denote the vector $(I^N,J^N,K^N)$. Again, if $H^N,|Z^N| \le N^{1/2}\log N$ then from \eqref{IJK} we find that $\mu(\J^N) = A \J^N + O((\log N)^2N^{-1/5})$. Letting $D = \diag(1,\gamma,\eta)$ denote the diagonal matrix with $1,\gamma,\eta$ along the main diagonal, and $\mathcal{V}^N$ denote the vector $(U^N,V^N,W^N)^{\top}$, so that $\V^N = D\J^N/H^N$. Moreover let $\Lambda = D \, A \, D^{-1}$ denote the conjugation of $A$ by $D$. Combining the estimate on $\mu(\J^N)$ with \eqref{muUVW} we find that if $t<\tau$ then
\begin{align}
\mu_t(\V^N) &= \frac{1}{H^N_t} D \mu_t(\J^N) + O((\log N)^2N^{-1/5} \nonumber \\
&= \Lambda \, \V^N_t + O((\log N)^2N^{-1/5} \label{muUVW2}
\end{align}
To reduce the dimensionality, since from \eqref{ILeq} we have $A(\alpha,\beta,1)^{\top}\R=0$ it follows that $\Lambda(\alpha,\beta\gamma,\eta)^{\top}\R=0$. We select the vector $\V_*=(u_*,v_*,w_*)$ with $u_*+v_*+w_*=1$, namely
\begin{align}
u_* = \alpha/d \qquad v_* = \beta \gamma/d \qquad w_* = \eta/d \qquad\hbox{where $d=\alpha+\beta \gamma +\eta.$}
\label{eq:uvw-lim}
\end{align}
Then, $W^N-w_* = -(U^N-u_*)-(V^N-v_*)$, so \eqref{muUVW2} can be re-written as
\begin{align}
\mu_t\begin{pmatrix} U^N_t-u_* \\ V^N_t-v_* \end{pmatrix}  =
\begin{pmatrix} \Lambda_{11}-\Lambda_{13} & \Lambda_{12}-\Lambda_{13}  \\
 \Lambda_{21}-\Lambda_{23} & \Lambda_{22}-\Lambda_{23} \end{pmatrix}
\begin{pmatrix} U^N_t-u_* \\ V^N_t-v_* \end{pmatrix} + O((\log N)^2N^{-1/5})
\label{muUVW3}
\end{align}
Note that $\Lambda$ is obtained from $A$ by multiplying entries in rows 2,3 by $\gamma,\eta$ and dividing entries in columns 2,3 by $\gamma,\eta$, respectively. Referring to \eqref{Adef} we find that
\begin{align}
\begin{pmatrix} \Lambda_{11}-\Lambda_{13} & \Lambda_{12}-\Lambda_{13}  \\
  \Lambda_{21}-\Lambda_{23} & \Lambda_{22}-\Lambda_{23} \end{pmatrix} = \begin{pmatrix} - (r_+y_* + 1 +r_-/\eta) & \left(\frac{2r_-}{\gamma} - \frac{r_-}{\eta} \right)  \\
 - (\gamma\lambda/\eta) & - (r_- + 2 +\gamma\lambda/\eta) \end{pmatrix}
\label{eq:Lbda-sml}
\end{align}
The diagonal entries in the matrix are negative, so the trace is negative. From \eqref{eq:eta-gamma2}, $\gamma/2<\eta$ which implies $2r_-/\gamma - r_-/\eta>0$, so the determinant is positive and so both eigenvalues have negative real part (which we already knew from analyzing $A$). To turn these calculations into control on the distance of $\V$ from $\V_*$ we let
\begin{align}\label{eq:thetas}
\theta_1 = 2r_-/\gamma - r_-/\eta \quad \text{and} \quad \theta_2=\gamma\lambda/\eta
\end{align}
and examine $Q^N_t = \theta_2 (U_t^N-u_*)^2 + \theta_1 (V_t^N-v_*)^2$.\\

\noindent\textbf{Approach.} First we show that w.h.p., $\tau^-_{(N^{-1/6}/2)}(Q^N) \wedge \tau \le C_{\ref{ray}}\log N$. From Lemma \ref{lem:qac-prod}, for a process $X$ we have $\mu(X^2)=2X \mu(X) + \sigma^2(X)$, and of course, $\sigma^2(X-c)=\sigma^2(X)$ and $\mu(X-c)=\mu(X)$ for any constant $c$. As noted in the proof of Lemma \ref{bdHI}, $U^N=I^N/H^N$ jumps by $O(1/H^N)$ at rate $O(H^N)$, so $\sigma^2(U^N)=O(1/H^N)$ and similarly for $V^N$. Let $a_1 = \min(r_+y_*+1+r_-/\eta,r_-+2+\gamma\lambda/\eta)$, so that both diagonal entries in \eqref{eq:Lbda-sml} are at most $-a_1$. Since the cross-terms cancel (by choice of $\theta_1,\theta_2$), from \eqref{muUVW3} we find that for $t<\tau$,
\begin{align*}
\mu_t(Q^N) &\le -2a_1\theta_2(U^N-u_*)^2 - 2a_1\theta_1(V^N-v_*)^2 + O((\log N)^2N^{-1/5}) + O(1/H^N) \\
&\le -2a_1 Q^N + a_2 (\log N)^2N^{-1/5} 
\end{align*}
for some constant $a_2$ and large $N$. It is also not hard to check that $|\mu_t(Q^N_t)| = O(Q^N_t)$ for $t<\tau$, which we will need in a moment. Letting $\tilde Q = Q - (a_2/a_1)(\log N)^2N^{-1/5}$, $\mu_t(\tilde Q^N) \le -2a_1 \tilde Q^N$ for $t<\tau$ so $\xi_t = \exp(2a_1 (t \wedge \tau))\tilde Q^N(t \wedge \tau)$ is a supermartingale. Since $(a_2/a_1)(\log N)^2N^{-1/5} \le N^{-1/6}/4$ for large $N$ and since $\xi_0 = \tilde Q^N_0 \le Q^N_0 \le \theta_1+\theta_2$,
$$P(\tau^-_{N^{-1/6}/2}(Q^N) \wedge \tau > t) \le P(\xi_t \ge e^{2a_1 t} N^{-1/6}/4) \le 4N^{1/6}e^{-2a_1t}(\theta_1+\theta_2).$$
Choosing $t$ equal to a large enough multiple of $\log N$, the RHS is $o(1)$.\\

\noindent\textbf{Control.} Now we suppose that $c_N^Q$ are constants with $N^{-1/6} \le c_N^Q = o(1)$ and show that
\begin{center}
if $Q^N_0 \le c_N^Q/2$ then w.h.p.\,$Q^N_t \le c_N^Q$ for all $t\le N \wedge \tau$.
\end{center}
We will use Corollary \ref{cor:driftbar}. Let
\begin{center}
$\tau_1 = \smash{\tau\wedge \tau^+_{c_N^Q}(Q^N)}$, \ $x=c_N^Q/2$ \ and \ $X_t = Q^N_{t\wedge \tau_1}-x$.
\end{center}
As noted above, $U^N$ jumps by $O(1/H^N)$ when $H^N=\omega(1)$. Thus when $H^N=\omega(1)$, $(U^N-u_*)^2$ jumps by $O((U^N-u_*)/H^N + 1/(H^N)^2)$. If $t<\tau$ and $Q^N \le c_N^Q$ then since $H^N\ge N^{1/5}$ and $c_N^Q\ge N^{-1/6}=\omega(1/H^N)$,
$$\frac{U^N-u_*}{H^N} + \frac{1}{(H^N)^2)} = O\left(\frac{(c_N^Q)^{1/2}}{H^N}\right ) = O(N^{-1/5}(c_N^Q)^{1/2}).$$
An analogous estimate holds for $(V-v_*)^2$, which shows that $\Delta_\star(X) = O(N^{-1/5}(c_N^Q)^{1/2}) = o(x)$. If $Q^N_t \ge c_N^Q/2$ and $t<\tau$, we have $\mu(Q^N) \le - (2a_1-o(1))c_N^Q/2 \le -\mu_\star$ for large $N$, where $\mu_\star = a_1c_N^Q/2$. As noted above, if $t<\tau$ then $|\mu(Q^N_t)| = O(Q^N_t)$ so if in addition $Q^N_t \le c_N^Q$ then $|\mu(Q^N_t)| = O(c_N^Q) \le C_{\mu_\star} = Cc_N^Q$ for large enough constant $C>0$. 
$Q^N$ has transition rate $O(H^N)$, and as shown above, if $Q^N \le c_N^Q$ and $t<\tau$ then $Q^N$ jumps by $O((c_N^Q)^{1/2}/H^N)$ and so $\sigma^2(Q^N_t) = O(c_N^Q/H^N) \le \sigma_\star^2 = Cc_N^QN^{-1/5}$ for large enough $C>0$. Since $\Delta_\star(X)\mu_\star/\sigma_\star^2 = O((c_N^Q)^{1/2})=o(1)$, let $C_\Delta=1$. Since $\mu_\star x/\sigma_\star^2 = \Omega(c_N^QN^{1/5}) = \Omega(N^{1/5-1/6})$, we find $\Gamma = \exp(\Omega(N^{1/30})) = \omega(N)$ and $x/16C_{\mu_\star} = \Omega(1)$. Corollary \ref{cor:driftbar} then gives the desired result.

\clearp

\section{Step 4: averaging $z_Ni_N$ to 0} \label{sec:aveto0}

In this section we prove Lemma \ref{aveto0}.
Letting $\sigma_z^2=4r_-(1-y_*)$, using \eqref{eq:z-trans} and recalling $r_-(1-y_*) = r_+y^2_*$, we see
 \begin{align*}
 z^N \to z^N + 2/N^{1/2} \quad \hbox{at rate} & \quad q_+ = \sigma_z^2N/8 - r_-z^NN^{1/2}/2\\
 z^N \to z^N - 2/N^{1/2} \quad \hbox{at rate} & \quad q_- = \sigma_z^2N/8 + r_+y_* z^NN^{1/2} + O(1 \vee (z^N)^2).
 \end{align*}
In order to prove an averaging result we'd like to work with a process whose transitions are symmetric on reflection about 0. Thus we define the following process $\tld z^N$, which can be thought of as a spatially discrete Ornstein-Uhlenbeck process. Letting $\mu_z = r_- + 2r_+y_*$ we let $\tld z^N$ have transitions
 \begin{align}\label{eq:tldz}
 \tld z^N \to \tld z^N + 2/N^{1/2} \quad \hbox{at rate} & \quad \tld q_+ = \tld q_+(\tld z^N)=\sigma_z^2N/8 - \mu_z N^{1/2} \tld z^N/4 \\
 \tld z^N \to \tld z^N - 2/N^{1/2} \quad \hbox{at rate} & \quad \tld q_- = \tld q_-(\tld z^N)=\sigma_z^2N/8 + \mu_z N^{1/2} \tld z^N/4 . \nonumber 
 \end{align}
Furthermore take $\tld z^N(0) = 2N^{-1/2}\lfloor N^{1/2}z^N(0)/2 \rfloor$ so that $\tld z^N$ takes values in $2N^{-1/2}\ZZ$. Couple $z^N$ with $\tld z^N$ in the obvious way, i.e.~couple jumps of $+2/N^{1/2}$ at the minimum of the two rates and similarly for jumps of $-2/N^{-1/2}$, and let $D^N = z^N-\tld z^N$ with respect to this coupling. The following result controls the size of $D^N$. The power of $\log N$ in the bound is not optimal but it's good enough and frees us from having to track yet another constant.

\begin{lemma}\label{lem:D}
With high probability,
$$\sup\{ \ |D^N_t| \ \colon t \le N \wedge \tau^+_{C_{\ref{maxz}}\sqrt{\log N}}(|z^N|) \ \}\le N^{-1/4}\log N.$$
\end{lemma}

\begin{proof}

$D^N$ has the following transitions:
 \begin{align*}
D^N \to D^N + 2/N^{1/2} \quad \hbox{at rate} & \quad q_+(D^N) = \max(q_+-\tld q_+,0) + \max(\tld q_- - \tld q_-,0) \\
D^N \to D^N - 2/N^{1/2} \quad \hbox{at rate} & \quad q_-(D^N) = \max(\tld q_+ - q_+,0) + \max(q_- - \tld q_-,0).
 \end{align*}
Using the fact that $\max(a,0) - \max(-a,0) = a$ for $a \in \R$, we compute
\begin{align*}
\mu(D^N) &= 2N^{-1/2}(q_+(D^N) - q_-(D^N)) \\
&= 2N^{-1/2}((q_+-\tld q_+) - (q_--\tld q_-))
\end{align*}
From their definition and the choice of constant $\mu_z=r_-+2r_+y_*$,
\begin{align*}
& ((q_+-\tld q_+) - (q_--\tld q_-)) \\
&= q_+-q_- - (\tld q_+-\tld q_-) \\
&= (-r_--2r_+y_*)z^NN^{1/2}/2 - (-\mu_z/4 - \mu_z/4)\tld z^NN^{1/2}/2 + O(1\vee (z^N)^2) \\
&= -\mu_z(z^N-\tld z^N)N^{1/2}/2 + O(1\vee (z^N)^2)
\end{align*}
and so
\begin{align}\label{eq:D-drift}
\mu(D^N) &= -\mu_z D^N + O((1\vee (z^N)^2)/N^{1/2}).
\end{align}
Computing the diffusivity,
\begin{align}\label{eq:D-diff}
\sigma^2(D^N) &= (4/N)(|q_+-\tld q_+| + |q_--\tld q_-|) \\
&= O(\max(z^N,\tld z^N,(1\vee (z^N)^2)/N^{1/2})/N^{1/2}. \nonumber 
\end{align}
Define $\tau = \tau^+_{C_{\ref{maxz}}\sqrt{\log N}}(|z^N|) \wedge \tau^+_{N^{-1/4}\log N}(|D^N|)$, and observe that if $t<\tau$ then not only is $|z^N_t| \le C_{\ref{maxz}}\sqrt{\log N}$ but also
$$|\tld z^N_t| \le |z^N_t| + |D^N_t| \le C_{\ref{maxz}}\sqrt{\log N} + N^{-1/4}\log N \le 2C_{\ref{maxz}}\sqrt{\log N}$$
for large $N$. We use Corollary \ref{cor:driftbar} to control $|D^N|$. Let $x = \frac{1}{2}N^{-1/4}\log N$ and let $X = -x + |D^N - x|$, then $\Delta_\star(X) \le 2/N^{1/2} \le x/2$ for large $N$. Since $D^N$ jumps by at most $x$, if $|D^N| \ge x$ then $\mu(X) = \sgn(D^N)\mu(D^N)$. Suppose $X\ge x/2$ and $t<\tau$, then $|D^N_t| \ge 3x/2$, and since $(1\vee(z^N_t)^2)/N^{1/2}=o(D^N_t)$, from \eqref{eq:D-drift} we have $\mu(D^N_t)/D^N_t \le -\mu_z/2$. Thus, letting $\mu_\star=3x\mu_z/4$, we find that $\mu_t(X) \le -\mu_\star$. Now suppose that $X \le x$ and $t<\tau$, then $|D^N_t| \le 2x$, and using \eqref{eq:D-drift}-\eqref{eq:D-diff} and $|z^N_t|,\tld |z^N_t|=O(\sqrt{\log N})$, $|\mu_t(z^N)| \le 4\mu_zx$ and $\sigma^2_t(D^N) \le CN^{-1/2}\sqrt{\log N}$ for some $C>0$ and large $N$, so let $C_{\mu_\star} = 4\mu_zx$ and let $\sigma^2_\star = CN^{-1/2}\sqrt{\log N}$. Then, $\Delta_\star(X)\mu_\star/\sigma^2_\star = o(1)$ which allows us to take $C_\Delta=1$. Then, $\Gamma = \exp(\Omega((\log N)^{3/2})) = \omega(N)$ and $x/16C_{\mu_\star} = \Omega(1)$, so the result follows from Corollary \ref{cor:driftbar}.
\end{proof}

In the context of Lemma \ref{aveto0}, if $t<\tau_N = \tau_N(c_N^h,c_N^Q,0)$ then $h_N(t) \le c_N^h \le \log N$. Letting $L_N(t)$ denote $L(i_N(t),j_N(t),k_N(t))$, since $L(0,0,0)=0$, $i_N,j_N,k_N\le h_N$ and $L$ is Lipschitz with constant $c_L$, $|L_N(t)| \le c_Lc_N^h \le c_L(c_N^h)^{1/2}(\log N)^{1/2})$. Let $\tau_{N,D} = \tau_N \wedge \tau^+_{N^{-1/4}\log N}(|D_N(s)|)$, then
$$\sup_{t \le T}\left|\int_0^{t\wedge \tau_{N,D}}|D^N(s) L_N(s)|ds \right| = O(c_LN^{-1/4}(c_N^h)^{1/2}(\log N)^{3/2} T).$$
Since by definition $\tau_N \le \tau^+_{C_{{\ref{maxz}}}\sqrt{\log N}}(|z^N|)$, Lemma \ref{lem:D} implies that w.h.p.~$\tau_{N,D} = \tau_N$, so it is enough to prove Lemma \ref{aveto0} with $z^N$ replaced by $\tld z^N$. Thus, we will prove the following result.

\begin{lemma}\label{lem:tldaveto0}
Let $L_N(s)$ denote $L(i_N(s),j_N(s),k_N(s))$. In the context of Lemma \ref{aveto0} except with $\tld z^N$ in place of $z^N$, with high probability
$$
\sup_{t\le T } \left|\int_0^{t \wedge \tau_N} \tld z_N(s)L_N(s) \,ds \right|  = O(c_L(c_N^h)^{1/2}(N^{-1/4}\vee (c_N^hc_N^Q)^{1/2})\log^2 N ).
$$
\end{lemma}

We begin by estimating excursions of $\tld z^N$.

\begin{lemma}\label{lem:tldz}
Define $c_* = \sqrt{\sigma_z^2/4\mu_z}$. Let $\tau^*_0 = \inf\{t\colon \tld z^N_t=0\}$ and let
\begin{align*}
\tau^*_1 & = \inf\{ s > \tau^*_0 : |\tld z^N(s)| > c_* \}\\
\tau^*_2 & = \inf\{ s > \tau^*_1 : \tld z^N(s) = 0 \}.
\end{align*}
\begin{itemize}[noitemsep]
\item If $C>1/\mu_z$ then w.h.p.~$\tau^*_0 \le C\log N$.
\item $E[\tau^*_1-\tau^*_0] \ge 1/4\mu_z$.
\item There are constants $\theta,\Theta>0$ so that $E[\exp(\theta(\tau^*_2-\tau^*_0))] \le \Theta$ for large $N$.
\item With $\theta,\Theta$ as above, $P\left(\displaystyle \int_{\tau^*_0}^{\tau^*_2}|\tld z^N_s|ds  > 3c_*M + (\theta+\sigma^2_*)M^2\right) \le (1+3\Theta)e^{-\theta m}$.
\end{itemize}
\end{lemma}

\begin{proof}
From \eqref{eq:tldz} we compute
$$\mu(\tld z^N) = -\mu_z \tld z^N \quad \text{and} \quad \sigma^2(\tld z^N) = \sigma_z^2.$$
If $\tld z^N_0>0$ then $\xi_t = e^{\mu_z (t \wedge \tau^*_0)}\tld z^N_{t \wedge \tau^*_0}$ is a supermartingale. Since $\tld z^N_t \ge 2N^{-1/2}$ if $\tau^*_0>t$,
$$P(\tau^*_0 > t \mid \tld z^N_0=z) \le P(\xi_t \ge 2N^{-1/2}e^{\mu_z t} \mid \tld \xi_0 = z) \le \frac{\sqrt{N}z}{2}e^{-\mu_z t}.$$
By symmetry of $\tld z^N$, the same estimate holds if $-\tld z^N_0=z$. Since $|\tld z^N| \le N^{1/2}+2N^{-1/2}$, the first statement follows by taking $t=C\log N$ for $C>1/\mu_z$.\\

To prove the second statement, let $\tau = \tau^+_{c_*}(|\tld z^N|)$, and note that by the strong Markov property, $\tau^*_1-\tau^*_0$ is equal to distribution to $\tau$ conditioned on $\tld z^N_0=0$. We compute
\begin{align}\label{eq:tldzsq-drift}
\mu((\tld z^N)^2) &= 2\tld z^N \mu(\tld z^N) + \sigma^2(\tld z^N) \nonumber \\
&= -2\mu_z(\tld z^N)^2 + \sigma_z^2.
\end{align}
In particular, $\mu((\tld z^N)^2) \le \sigma_z^2$, so $(\tld z^N_{t\wedge \tau})^2 - \sigma_z^2(t\wedge \tau)$ is a supermartingale. Since $(\tld z^N_0)^2=0$ and $(\tld z^N_\tau)^2 \ge c_*^2$, using optional stopping
$$E[\tau] \ge E[\tld z^N_\tau]/\sigma_z^2 \ge c_*^2/\sigma_z^2 = 1/4\mu_z.$$

To prove the third statement it suffices to show that for large $N$, both $P(\tau^*_1-\tau^*_0>t)$ and $P(\tau^*_2-\tau^*_1>t)$ are bounded by some function that decays exponentially in $t$. We prove the two parts in the order given. Suppose $|\tld z^N_0| \le c_*$ and let $\tau = \tau^+_{c_*}(|\tld z^N|)$. From \eqref{eq:tldzsq-drift}, if $|\tld z^N| \le c_* = \sqrt{\sigma_z^2/4\mu_z}$ then $\mu((\tld z^N)^2) \ge \sigma_z^2/2$. Moreover if $| \tld z^N | \le c_*$ then $(\tld z^N)^2$ jumps by $2(2N^{-1/2})\tld z^N-(\tld z^N)^2 \le 4N^{-1/2}c_*$ for large $N$, and jumps at rate $\sigma_z^2N/4$, so $\sigma^2((\tld z^N)^2) \le 4c_*^2\sigma_z^2$. 
Let $X_t = (\tld z^N_{t\wedge \tau})^2 - (\tld z^N_0)^2$, so that for $\phi>0$,
$$-X_t^m - \phi \lng X \rng_t \ge - X_t + (\sigma_z^2/2 - 4c_*^2\sigma_z^2 \phi)(t \wedge \tau).$$
Let $\phi=1/16c_*^2$ so that there is $\sigma_z^2/4$ in the above parentheses, noting that $\phi\Delta_\star(X) =O(N^{-1/2}) \le 1/2$ for large $N$. Using Lemma \ref{lem:sm-est} with $a=1/\phi = 16c_*^2$, we find that
$$P(\tau>t \ \text{and} \ X_t \le - 16c_*^2 + \sigma_z^2 t/4) \le 2/e.$$
If $\tau>t$ then $|\tld z^N_t| \le c_*$, so $X_t \le 4(c_*)^2$. Denoting the constant $t^* = 80c_*^2/\sigma_z^2$ we find that
$$P(\tau > t^*) \le 2/e.$$
Using the Markov property to iterate, $P(\tau > kt^*) \le (2/e)^k$, which proves the first part.\\

Taking now $\tau = \inf\{t \colon \tld z^N=0\}$ and $j^*$ the least multiple of $2N^{-1/2}$ larger than $c_*$, we bound $P(\tau^*_2-\tau^*_1>t) = P(\tau >t \mid |\tld z^N_0|=2j^*/\sqrt{N})$. Since this quantity is non-decreasing in $j^*$ and we only need an upper bound we may assume $j^*$ is even, and by symmetry we may assume $\tld z^N>0$. For any $j$ let $\tau_j = \inf\{t \colon \tld z^N_t=2j/\sqrt{N}\}$ and let $\tau_{0,j} = \inf\{t \colon \tld z^N_t \in \{0,2j/\sqrt{N}\}\}$. The approach is to condition on the number of commutes from $2j^*/\sqrt{N}$ to $j^*/\sqrt{N}$ and back, before hitting $0$. Since $\tld z^N$, stopped at zero, is a supermartingale,
$$P(\tld z^N_{\tau_{0,j^*}}=0 \mid \tld z^N_0 = j^*/\sqrt{N}) \ge 1/2,$$
so the number of commutes is at most $\geo(1/2)$. It is easy to check that if a random variable $X$ has an exponential tail, then a geometric sum (with positive stopping probability) of i.i.d.~copies of $X$ itself has an exponential tail. Thus it suffices to show that both
$$P(\tau_{j^*/2}>t \mid \tld z^N_0 = 2j^*/\sqrt{N}) \ \text{and} \ P(\tau_{0,j^*} >t \mid \tld z^N_0= j^*/\sqrt{N})$$
are bounded by some function that decays exponentially in $t$, for large $N$. Using the supermartingale $\xi$ defined above, it is easy to check that
$$P(\tau_{j^*/2}>t \mid \tld z^N_0=2j^*/\sqrt{N}) \le 2e^{-\mu_z t}.$$
Then, since $\tau_{0,j^*} \le \inf\{t \colon |\tld z^N_t| \ge 2j^*/\sqrt{N}\}$, using the proof of the bound on $\tau^*_1-\tau^*_0$ we deduce that $P(\tau_{0,j^*}> kt^* \mid \tld z^N_0= j^*/\sqrt{N}) \le (2/e)^k$. Since $j^*$ is even, in applying the proof we may need to replace $c_*$ with a quantity up to $2N^{-1/2}$ larger, but the only effect is an $o(1)$ change in $t^*$.\\

To prove the fourth statement, we first note that
$$\int_{\tau^*_0}^{\tau^*_2}|\tld z^N_s|ds \le (\tau^*_1-\tau^*_0)c_* + \int_{\tau^*_1}^{\tau^*_2}|\tld z^N_s|ds.$$
Since $\tau^*_1\le \tau^*_2$, the third statement shows $P((\tau^*_1-\tau^*_0)c_*>c_*M) \le \Theta e^{-\theta M}$. To bound the second term we use the fact that 
$$\int_{\tau^*_1}^{\tau^*_2} |\tld z^N_s|ds \le (\tau^*_2-\tau^*_1)\sup_{t \in [\tau^*_1,\tau^*_2]}|\tld z^N_t|,$$ 
then control the latter quantity. For ease of notation we suppose $\tld z^N_0 = 2j^*/\sqrt{N}$, let $\tau=\inf\{t \colon \tld z^N_t=0\}$ and control $\sup_{t \le \tau}\tld z^N_t$. Define
$X_t = \tld z^N_{t \wedge \tau}$, then $X$ is a supermartingale with $X_0\le c_*+2/\sqrt{N} \le 2c_*$ for large $N$, and for $\phi>0$
$$X_t^m - \phi \lng X\rng_t \ge X_t-X_0-\sigma^2_*\phi(t\wedge \tau).$$
Using Lemma \ref{lem:sm-est}, if $2\phi/\sqrt{N} \le 1/2$ then $P(X_t > 2c_*+a+\sigma^2_*\phi(t \wedge \tau) \ \text{for some} \ t\ge 0) \le 2e^{-\phi a}$. Fix $\phi=1$. Since $\tau$ is equal in distribution in $\tau^*_2-\tau^*_1$, which is at most $\tau^*_2-\tau^*_0$, from the third statement $P(\tau>M) \le \Theta e^{-\theta M}$. Letting $a=\theta M$ we find
$$P(\sup_{s \le \tau}X_s > 2c_*+(\theta  + \sigma^2_*)M) \le (1+\Theta)e^{-\theta M}.$$
Combining with the estimate of the first term, and another estimate on $\tau^*_2-\tau^*_1$, we obtain the desired result.
\end{proof}

\begin{proof}[Proof of Lemma \ref{lem:tldaveto0}]
We begin by whittling down the statement of Lemma \ref{lem:tldaveto0} until it is ready to analyze in detail. Define
$$L_N^*(s) = \1(s<\tau_N)L_N(s) + \1(s \ge \tau_N)\lim_{t \ua \tau_N}L_N(t).$$
In words, $L^*_N$ is equal to $L_N$ until $\tau_N$, at which point it remains at the last value assumed by $L_N$ before $\tau_N$. Since $L_N^*(s)=L_N(s)$ for $s<\tau_N$, the statement of Lemma \ref{lem:tldaveto0} is unchanged if we replace $L_N$ with $L_N^*$.\\

In the same manner as $\tld z_N$ is coupled to $\tld z^N$ as described at the beginning of this section, it can be coupled to the full infection process. Recall the $c_N^h,c_N^Q$ control time $\tau_N = \tau_N(c_N^h,c_N^Q,\ep)$, defined in Definition \ref{def:ctrl-time} and given on the slow time scale. Define $\tau_{N,\tld z} = \inf\{t \colon |\tld z_N(t)| \ge 2C_{\ref{maxz}}\sqrt{\log N}\}$. By Lemma \ref{lem:D} and the definition of $\tau_N$, w.h.p.
$$\tau_{N,\tld z} \ge \tau^+_{C_{\ref{maxz}}\sqrt{\log N}}(|z_N|) \ge \tau_N,$$
so it is enough to show the statement of Lemma \ref{lem:tldaveto0} holds with $\tau_{N,\tld z}$ in place of $\tau_N$. As in Lemma \ref{lem:tldz} let $c_*=\sqrt{\sigma_z^2/4\mu_z}$. Define $\tau^*_0 = \inf\{t \colon \tld z_N(t)=0\}$ and for $j\ge 0$,
\begin{align*}
\tau^*_{2j+1} & = \inf\{ s > \tau^*_{2j} : |\tld z_N(s)|>c_* \} \\
\tau^*_{2j+2} & = \inf\{ s > \tau^*_{2j+1} : \tld z_N(s) = 0 \}.
\end{align*}
The definition of $\tau^*_i$, $i=0,1,2$ differs from Lemma \ref{lem:tldz} only by the choice of time scale. We show the contribution to the integral up to time $\tau^*_0$ can be ignored. By Lemma \ref{lem:tldz}, w.h.p.~$\tau^*_0 \le (2/\mu_z)N^{-1/2}\log N$. Since $i_N,j_N,k_N\le h_N$, $L(0,0,0)=0$ and $L$ has Lipschitz constant $c_L$ in the $\ell^1$ norm, $|L_N(s)| \le c_Lh_N(s)$. If $s<\tau_N$ then by definition $h_N(s)\le c_N^h$. Since $L^*_N$ only sees the values $\{L_N(s) \colon s<\tau_N\}$ it follows that $|L_N^*(s)| \le c_L c_N^h$ for all $s\ge 0$. Since $|\tld z^N_t|\le 2C_{\ref{maxz}}\sqrt{\log N}$ for $t<\tau_{N,\tld z}$, using the above estimate on $\tau^*_0$ and $c_N^h \le (c_N^h\log N)^{1/2}$, with high probability
\begin{align*}
\sup_{t \le T} \left|\int_0^{t \wedge \tau^*_0\wedge \tau_{N,\tld z}} z_N(s)L^*_N(s) \,ds \right| 
& \le 2C_{\ref{maxz}}\sqrt{\log N}c_Lc_N^h(2/\mu_z)N^{-1/2}\log N \\
& = O(c_L(c_N^h)^{1/2} N^{-1/2}\log^2N).
\end{align*}
Thus, without affecting the conclusion, we may assume that $\tau^*_0=0$, equivalently, $\tld z_N(0)=0$. At this point we will also replace $t\wedge \tau_{N,\tld z}$ with $t$ as the upper endpoint; this does not decrease the supremum over $t\le T$. To summarize thus far, it remains to show that, if $\tld z_N(0)=0$ and $T>0$ then w.h.p.,
\begin{align}\label{eq:aveint}
\sup_{t\le T } \left|\int_0^t \tld z_N(s)L_N^*(s) \,ds \right|  = O(c_L(c_N^h)^{1/2}(N^{-1/4}\vee (c_N^hc_N^Q)^{1/2})\log^2 N ).
\end{align}
For $j,k\ge 0$ define $\xi_j = \int_{\tau^*_{2j}}^{\tau^*_{2j+2}} \tld z_N(s) \, ds$,
$$
S_k = \sum_{j=0}^k \xi_j L_N^*(\tau^*_{2j}) \quad \text{and} \quad G_k = \int_{\tau^*_{2k}}^{\tau^*_{2k+2}} |\tld z_N(s)(L_N^*(s)-L_N^*(\tau^*_{2k}))|ds.
$$
Let $\I$ denote the LHS of \eqref{eq:aveint}. If $K_N$ is such that w.h.p.~$\tau^*_{2K_N} \ge T$, then w.h.p.
\begin{align}\label{eq:aveint-terms}
\I \le |\sup_{k < K_N}S_k| + \sum_{k < K_N} G_k.
\end{align}
Since the variables $\{\tau^*_{2j+2}-\tau^*_{2j}\}_{j\ge 0}$ are i.i.d., by Lemma \ref{lem:tldz},
$$E[\tau^*_{2j+2}-\tau^*_{2j}] \ge 1/(4\mu_z\sqrt{N})=\ep/\sqrt{N},$$
and since for random variable $X\ge 0$, $E[\theta^2 X^2/2] \le E[e^{\theta X}]$, by Lemma \ref{lem:tldz},
$$\var(\tau^*_{2j+2}-\tau^*_{2j}) \le E[(\tau^*_{2j+2}-\tau^*_{2j})^2] \le (2\Theta/\theta^2)/N = C/N.$$
Letting $K_N = 2T/E[\tau^*_2-\tau^*_0] = 2\sqrt{N}T/\ep$ we have $E[\tau^*_{2K_N}] = 2T$ and $\var(\tau^*_{2K_N}) \le K_N C/N = 2CT/\ep\sqrt{N}$, so using Chebyshev's inequality,
$$P(\tau^*_{2K_N} < T) \le \frac{\var(\tau^*_{2K_N})}{T^2} = O(1/\sqrt{N}),$$
and $\tau^*_{2K_N} \ge T$ w.h.p., as desired. It remains to estimate the terms in \eqref{eq:aveint-terms}.

For $t\ge 0$ let $\F(t)$ denote the information up to time $t$. The $\{\xi_j\}_{j\ge 0}$ are i.i.d., and by symmetry of $\tld z^N$, $\xi_j$ and $-\xi_j$ are equal in distribution. Because of this and since $L_N^*(\tau^*_{2j})$ is $\F(\tau^*_{2j})$-measurable, $S$ is a discrete time martingale. Using Doob's $L^2$ maximal inequality for martingales and orthogonality of martingale increments,
\begin{align*}
E [ ( \sup_{k < K_N} S_k )^2 ]
& \le 4 E \bigg[ \bigg( \sum_{j=0}^{K_N-1} \xi_j L^*_N(\tau^*_{2j}) \bigg)^2 \bigg] \le c_L^2(c_N^h)^2 \sum_{j=0}^{K_N-1} E[ \, \xi^2_j \, ]
\end{align*}
From Lemma \ref{lem:tldz} and noting the change of time scale, there are $\theta,C>0$ so that for large $M$, $P( \big ( \int_{\tau_0^*}^{\tau_{2}^*} |\tld z_N(s)|ds \big)^2 > CM^4/N) \le Ce^{-\theta M}$, which implies that
\beq
E[ \, \xi_j^2 \, ] \le E\bigg[ \big( \int_{\tau_0^*}^{\tau_{2}^*} |\tld z_N(s)|ds \big)^2 \bigg] \le C/N
\label{2ndterm}
\eeq
for some possibly larger $C>0$. It follows that
\begin{align}\label{pwc}
E [ \, |\sup_{k < K_N} S_k| \, ] & \le (E[ \, (\sup_{k < K_N}S_k)^2 \, ])^{1/2} \nonumber \\
& = O(c_Lc_N^h(K_N/N)^{1/2}) = O(c_Lc_N^hN^{-1/4}).
\end{align}
Noting that $c_N^h \le (c_n^h \log N)^{1/2}$, then using Markov's inequality we find that
$$\text{w.h.p.} \ \sup_{k < K_N}S_k = O(c_L(c_N^h)^{1/2}N^{-1/4}\log N).$$
It remains to estimate $\sum_{k < K_N}G_k$. Let $\overline \xi_k$ denote $\int_{\tau_{2k}^*}^{\tau_{2k+2}^*} |\tld z_N(s)|ds$. Then
$$
G_k \le \, \overline \xi_k \, \sup_{s\in [\tau^*_{2k},\tau^*_{2k+2}]} |L^*_N(s) - L^*_N(\tau^*_{2k})|
$$
so by the Cauchy-Schwarz inequality
\beq
E[ G_k ] \le ( E[ \, \overline\xi_k^2 \, ])^{1/2} \bigg( E \big[ \big(\sup_{s\in [\tau^*_{2k},\tau^*_{2k+2}]} |L^*_N(s) - L^*_N(\tau^*_{2k})| \big)^2 \big] \bigg)^{1/2}
\cdot 
\label{eq:Gk-mean}
\eeq
The first term is $O(1/\sqrt{N})$ by the latter part of \eqref{2ndterm}. Let $M_k^L$ denote the supremum inside the second term. To estimate $E[(M_k^L)^2]$, recall \eqref{eq:uvw-lim} and note that if $s<\tau_N$ then
$$(i_N(s), \gamma j_N(s), \eta k_N(s))=(u_*,v_*,w_*)h_N(s) + O((c_N^Q)^{1/2}c_N^h).$$
It follows that
$$M_k^L = O(c_L \sup_{\tau^*_{2k} \le s < \tau^*_{2k+2}}|h_N(s \wedge \tau_N)-h_N(\tau^*_{2k})|) + O(c_L(c_N^Q)^{1/2}c_N^h).$$
Let $M_k^h$ denote the above supremum. Using the simple inequality $(a+b)^2 \le 2(a^2+b^2)$,
\begin{align}\label{eq:mkl}
E[(M_k^L)^2] &= O(c_L^2E[(M_k^h)^2] + c_L^2c_N^Q(c_N^h)^2).
\end{align}
We estimate $M_k^h$. If $s<\tau_N$ then since $|z_N(s)| \le C_{\ref{maxz}}\sqrt{\log N}$ and $h_N(s) \le c_N^h$, referring to \eqref{fasthdrift} we have
$$
\mu_s(h_N) = O(\sqrt{\log N}c_N^h + (c_N^h)^2) = O(c_N^h\log N)
$$
and so
$$
\int_{\tau^*_{2k}}^{s \wedge \tau_N}|\mu_r(h_N)| \, dr = (s \wedge \tau_N-\tau^*_{2j}) O(c_N^h\log N).
$$
Since $h_N$ jumps by $1/\sqrt{N}$ at rate $O(Nh_N)$, if $s<\tau_N$ then $\sigma^2_s(h_N) = O(c_N^h)$.\\
Applying Doob's $L^2$-maximal inequality,
\begin{align*}
&E\bigg[\big( \sup_{\tau^*_{2k} \le s \le t \wedge \tau_N} h_N(s) - h_N(\tau^*_{2k}) - \int_{\tau^*_{2k}}^{s \wedge \tau_N} \mu_r(h_N) \, dr \big )^2 \mid \F_{\tau^*_{2k}} \bigg] \\
& \le 4 E \bigg[ \big( h_N(t \wedge \tau_N) - h_N(\tau^*_{2k}) - \int_{\tau^*_{2k}}^{t \wedge \tau_N} \mu_r(h_N) \, dr \big)^2 \mid \F_{\tau^*_{2k}} \bigg] \le (t -\tau^*_{2k})O(c_N^h). \notag
\end{align*}
Using again the inequality $(a+b)^2 \le 2(a^2 + b^2)$, from the above two displays we obtain
$$
E\bigg[ \bigg( \sup_{\tau^*_{2k} \le s \le t \wedge \tau_N} |h_N(s) - h_N(\tau^*_{2k})| \mid \F_{\tau^*_{2k}} \bigg)^2 \bigg] \le (t-\tau^*_{2k})O(c_N^h) + (t-\tau^*_{2k})^2 O((c_N^h)^2\log^2N)
$$
Letting $t = \tau^*_{2k+2}$ and taking an expectation we find that
\begin{align*}
E[ \, (M_k^h)^2 \, ] & \le O(c_N^h/\sqrt{N}) + O(N^{-1}(c_N^h)^2\log^2N) \\
&= O(c_N^h/\sqrt{N}). 
\end{align*}
Using \eqref{eq:mkl} we find that
$$(E[(M_k^L)^2])^{1/2} = O(c_L(c_N^h)^{1/2}(N^{-1/4} \vee (c_N^hc_N^Q)^{1/2}).$$
Recalling that $K_N = O(\sqrt{N})$ and $(E[\overline \xi_k^2])^{1/2}=O(1/\sqrt{N})$, from \eqref{eq:Gk-mean} we have
$$E\big[\, \sum_{k < K_N}G_k \, \big] = O(c_L(c_N^h)^{1/2}(N^{-1/4} \vee (c_N^hc_N^Q)^{1/2}).$$
Again, using Markov's inequality we find that
$$\text{w.h.p.} \ \sum_{k < K_N}G_k = O(c_L(c_N^h)^{1/2}(N^{-1/4}\vee (c_N^hc_N^Q)^{1/2})\log N)$$
and the proof is complete.
\end{proof}

\clearp

\section{Step 5: bounding the time for $h_N$ to reach $O(1)$ } \label{sec:Hdown}

In this section our goal is to prove Lemma \ref{downfast}. Before proving Lemma \ref{downfast}, we need one additional estimate, that controls the transient behaviour of $z^N$.

\begin{lemma}\label{lem:zint}
 Let $\tau = \inf\{t:|z_t^N| \le C_{\ref{maxz}}N^{-1/2}\}$. Then with high probability
 $$\int_0^{\tau}|z_t^N| dt= O(|z_0^N|) + O((\log N)^2).$$
\end{lemma}

\vskip10pt

\begin{proof}
We use Lemma \ref{lem:driftbar}, and assume $z_0^N>0$; the case $z_0^N<0$ is analogous. Let $ \mu_\star= x =\log N$. Define $X_t = z^N_t - z^N_0 - \int_0^t\mu_s(z^N)ds - \mu_\star t$. We note that $X_0=0$, $\mu(X) = -\mu_\star$ and $\sigma^2(X) = \sigma^2(z^N) = O(1)$, by \eqref{sigdiff}, so let $\sigma_\star^2$ be a large enough constant.
Since $\Delta_\star(X)=2N^{-1/2},$ $\Delta_\star(X)\mu_\star/\sigma_\star^2=o(1)$, so we can take $C_\Delta=1$. Since $|\mu(X)|=\mu_\star$, let $C_{\mu_\star}=\mu_\star$. With these choices $\Gamma = \exp(\Omega(\mu x))=\omega(N)$ and $x/C_{\mu_\star} = \Omega(1)$. By Lemma \ref{lem:driftbar}, w.h.p.~$X_t \le x$ for all $t \le N$.

Next we show this implies the desired bound. Since $z^N_t>0$ for $t<\tau$ by assumption, $\mu_t(z^N) \le -rz^N_t$, with $r$ as in the proof of Lemma \ref{maxz}. Thus if $X_t \le x$ then
$$z^N_t \le z^N_0 - r \int_0^tz^N_sds + (t+1)x.$$
Solving by repeated substitution we find $z^N_t \le z^N_0e^{-rt} + x((1-e^{-rt})/r + e^{-rt}) = z^N_0e^{-rt} + O(x)$. So $\int_0^tz^N_sds = O(z^N_0) + O(xt)$. By Lemma \ref{maxz}, $\tau \le C_{\ref{maxz}}\log N$ w.h.p., and the result follows.
\end{proof}

\vskip10pt

\begin{proof}[Proof of Lemma \ref{downfast}]
Recall the goal is to estimate $\tau = \tau^-_C(h_N) = \inf\{t \colon h_N(t) \le C\}$, as $C\to\infty$, assuming $h_N(0)=\omega(1)$. 
We may assume $C\ge 1$ so that $h_N(t) \ge 1$ for $t<\tau$. Recall from  \eqref{fasthdrift} that
$$
\mu(h_N) = c_1 z_N i_N - c_2i_N^2 + c_3 i_N/\sqrt{N},
$$
where $c_i, i=1,2,3$ are positive constants, and that $\sigma^2(h_N) = O(h_N)$. 
Using the Taylor approximation of Lemma \ref{lem:taylor}, if $h_N=\omega(N^{-1/2})$ then
\begin{align*}
 \left|\mu\left(\frac{1}{h_N}\right) + \frac{1}{h_N^2}\,\mu(h_N)\right| = \sigma^2(h_N)O\left(\frac{1}{h_N^3}\right) = O\left(\frac{1}{h_N^2}\right),
 \end{align*}
 and combining with the previous display,
 \begin{align*}
 \mu\left(\frac{1}{h_N}\right) = -\frac{c_1z_Ni_N}{h_N^2} + c_2\frac{i_N^2}{h_N^2} + o(1) + O\left(\frac{1}{h_N^2}\right).
\end{align*}
Let $x(t)=1/h_N(t)$ and $\nu=1/C$, so that $\tau=\inf\{t:x(t) \ge \nu\}$ and $x(t) <\nu$ for $t<\tau$. Since $h_N(0)\to\infty$, $x(0)=o(1)$. From the above display, if $t<\tau$ then
\begin{align}\label{eq:xdrift}
x(t)^p = -\int_0^{t}\frac{c_1z_N(s)i_N(s)}{h^2_N(s)}ds + \int_0^{t} c_2i^2_N(s)/h^2_N(s)ds + o(t) +\int_0^t O\left(\frac{1}{h_N^2}\right) dt. 
\end{align}
Since $i_N\le h_N$ and $h_N(t)\ge 1/\nu$ for $t<\tau$, for large $N$
\begin{align}\label{eq:x-compen}
x(t \wedge \tau)^p \le (c_1/\nu)\int_0^t|z_N(s)|ds + c_4(t\wedge \tau)
\end{align}
where $c_4$ is another positive constant. Let $\tau_1 = \inf\{t \colon |z_N(t)| \le C_{\ref{maxz}}N^{-1/2}\}$. Since $z_N$ has constant sign on $[0,\tau_1]$, using Lemma \ref{lem:zint} and recalling that $|z_0| = o(N^{1/2})$ by assumption, w.h.p.
\begin{align}\label{eq:zint-tau1}
\int_0^{\tau_1} |z_N(s)| ds = O(N^{-1/2}|z_0|) + O(N^{-1/2}(\log N)^2) = o(1).
\end{align}
To estimate the integral on $s \in [\tau_1,t]$, couple $z_N$ to $\tld z_N$ beginning at time $\tau_1$ so that $|D_N(\tau_1)| \le 2N^{-1/2}$. Noting that $z_N(\tau_1) \le C_{\ref{maxz}}N^{-1/2}$, by Lemma \ref{maxz} and \ref{lem:D}, w.h.p. $|D_N(s)| \le N^{-1/4}\log N$ for all $s \in [\tau_1,t]$. Let $\tau^*_k$ be as in the proof of Lemma \ref{lem:tldaveto0} except beginning with $\tau^*_0 = \tau_1$. Since $\tld z_N(\tau_1) \in [0,c_*]$, the value of $\tau^*_2-\tau^*_0$ is not larger than if $\tld z_N(\tau^*_0)=0$. Following that proof, if $M_1>0$ is a large enough constant then w.h.p. as $N\to\infty$ for fixed $t$, $\tau^*_{2\lfloor M_1\sqrt{N}t \rfloor} \ge t$. Since $E[ \, |\xi_j| \, ]=O(1/\sqrt{N})$, letting $M_2 = M_1\limsup_n(\sqrt{N}E[ \, |\xi_j| \, ])$, it follows that for large $N$,
$$E\big[\sum_{k \le M_1\sqrt{N}t} \int_{\tau^*_{2k}}^{\tau^*_{2k+2}}|\tld z_N(s)|ds \big] \le M_2t.$$
Using Markov's inequality on the last display, combining with the bound on $|D_N|$, and using the fact that $\tau^*_{2\lfloor M_1\sqrt{N}t \rfloor} \ge t$ w.h.p.,
$$P\big ( \ \int_{\tau_1}^t |z_N(s)| ds > (N^{-1/4}\log N + M_2/2\ep)t \ \big) \le \ep/2 + o(1).$$
Combining with \eqref{eq:zint-tau1}, we find that with probability $\ge 1-\ep/2-o(1)$,
\begin{align*}
x(t \wedge \tau)^p \le (c_4 + c_1(M_2/2\ep + o(1))/\nu )t.
\end{align*}
Since $1/h_N$ jumps by $O(N^{-1/2}/h_N^2)$ at rate $O(Nh_N)$, $\sigma^2(x) = \sigma^2(1/h_N) \le c_5/h_N^3 = c_5x^3$ for some $c_5>0$. Using Lemma \ref{lem:sm-est}, if $a>0$ and $\Delta_\star(x)\phi \le 1/2$ then with probability at least $ 1-2e^{-\phi a}$, for all $t\ge 0$
\begin{align}\label{eq:x-sm-est}
|x(t\wedge \tau) - x(0) - x(t \wedge \tau)^p| \le a + \phi c_5 \nu^3 (t \wedge \tau).
\end{align}
Using the above bound on $x(t)^p$, noting $x(0)=o(1)$ (since $h(0)\to\infty$) and $\nu \le 1$, with probability $\ge 1-2e^{-\phi a}-\ep$ for large $N$,
$$x(t \wedge \tau) \le (c_4 + c_1(M_2/2\ep+o(1))/\nu + c_5\phi)t + a+o(1).$$
Let $\ep=\nu$, $a=\nu/4$, $\phi=1/\nu^2$, let $M_3 = 2(c_4+c_1M_2 + c_5)$ and let $t=\nu^3/M_3$. If $N$ is large then $M_2/2\ep+o(1) \le M_2/\ep= M_2/\nu$. Since $\nu \le 1$ by assumption,
$$x(t \wedge \tau) \le \nu/2 + \nu/4 + o(1) < \nu$$
for large $N$, implying $\tau > t$. Since $\phi a = 1/4\nu$ and $\ep=\nu$, we have shown that
$$\lim_{\nu\to 0^+}\limsup_N P(\tau \le \nu^3/M_3) = 0.$$

To obtain the upper bound we need to take one more term in the Taylor series for $1/h_N$. Expanding to third order and noting that $1/h_N$ jumps by $O(N^{-1/2}/h_N^2)$ at rate $O(Nh_N)$,
\begin{align}\label{eq:xdrift2}
\mu\left(\frac{1}{h_N}\right) = -\frac{c_1z_Ni_N}{h_N^2} + c_2\frac{i_N^2}{h_N^2} +\frac{1}{h_N^3}\sigma^2(h_N)  + O({N^{-1/2} h_N^{-5}}).
\end{align}
Since we only need an upper bound on $\tau$, we wait for an amount of time $\nu \wedge \tau$ before estimating the compensator. To estimate the first term we note that $L(i_N,h_N) = i_N/h_N^2$ is Lipschitz with constant $c_L=2$, when $h_N\ge 1$. By Lemma \ref{notransient}, we may use Lemma \ref{aveto0} with $c_h=\log N$ and $c_Q=N^{-1/6}$ to find that w.h.p.,
$$\sup_{\nu \le t \le 1}\big|\int_{\nu\wedge \tau}^{t \wedge \tau}\frac{c_1z_N(s)i_N(s)}{h_N^2(s)}ds \big| = O(N^{-1/12}\log^3 N)=o(1).$$
By Lemma \ref{bdHI}, w.h.p.~$c_2i_N(t)^2/h_N(t)^2\ge c_6$ for all $\nu \wedge \tau \le t \le N^{1/2}\wedge \tau$ (since $H_t \ge N^{1/5}$ for $t \le \tau$), and some $c_6>0$. Combining these observations with \eqref{eq:xdrift2} we find that w.h.p.,
$$\int_\nu^{t\wedge \tau} \mu_s(x) \ge (c_6-o(1))(t\wedge \tau -\nu\wedge \tau).$$
Recalling that $\sigma^2_t(x) \le c_5\nu^3$ for $t<\tau$, using Lemma \ref{lem:sm-est} and the above display, if $a>0$ and $\Delta_\star(x)\phi \le 1/2$ then with probability at least $1-2e^{-\phi a}-o(1)$, for $t\ge \nu$,
\begin{align*}
x(t\wedge \tau) & \ge x(\nu \wedge \tau) + \int_{\nu \wedge \tau}^{t \wedge \tau} \mu_s(x)ds - a - \phi c_5 \nu^3 (t \wedge \tau - \nu\wedge \tau) \\
& \ge ((c_6-o(1)-\phi c_5\nu^3)(t \wedge \tau - \nu \wedge \tau) - a.
\end{align*}
On the above event, if $\tau>t \ge \nu$ then $t\wedge \tau=t$, $\nu\wedge \tau=\nu$ and $x(t\wedge \tau)=x(t)<\nu$. Taking $\phi=\nu^{-2}$ and $a=\nu$, if $\nu$ is small enough and $N$ large enough that $c_5\nu +o(1) \le c_6/2$, then
$$\nu \ge (c_6/2)(t-\nu)-\nu$$
and so $t \le (1+2/c_6)\nu$. Since our choice of $\phi,a$ gives $e^{-\phi a}=e^{-1/\nu}\to 0$ as $\nu \to 0$, it follows that
$$\lim_{\nu \to 0^+}\limsup_N P(\tau > (1+2/c_6)\nu)=0.$$
\clearp
\end{proof}

\section{Small values of $H^N$}\label{sec:T0bd}

In this section our goal is to prove Lemmas \ref{lem:Hdown} and \ref{lem:Hdn1}. First we prove Lemma \ref{lem:Hdn1}. Thus, given $\eps>0$, we suppose $h_N(0) \le \dlt $ where $\delta$ is to be determined during the course of the proof.  Let $a_k = 2^k \dlt$ for integer $k$ and recall $\tau_N(\ctl,0)$, given by Definition \ref{def:ctrl-time}. To prove Lemma \ref{lem:Hdn1} we will repeat the following step up to time $\tau_N(\ctl,0)$: start from $a_k + O(N^{-1/2})$ and run the process up to time $\tau_k^* = \tau_k\wedge \tau_N(\ctl,0)$, where $\tau_k$ is the first exit time of $h_N$ from $(a_{k-1},a_{k+1})$. Since jump sizes are $O(N^{-1/2})$, if $\tau_k^*=\tau_k$ then $h_N(\tau_k^*) = a_j + O(N^{-1/2})$ for some $j \in \{k-1,k+1\}$. Since $H_N(t)\ge N^{1/5}$ for $t<\tau_N(\ctl,0)$, $a_k=\Omega(N^{-0.3})$ for the duration of this iterative procedure. We begin by estimating the expectation of the exit times.

\begin{lemma} \label{lem:exit-exp}
Fix a positive integer $m>0$. There is a constant $C_{\ref{lem:exit-exp}}$ so that if $4^m \delta$ is small enough, $k\le m$ and $a_k \ge 2N^{-0.3}$, if $h_N(0) \in (a_{k-1},a_{k+1})$ then $E[\tau_k^*] \le C_{\ref{lem:exit-exp}} a_k$ for large $N$.
\end{lemma}

\vskip10pt

\begin{proof}
We claim that to complete the proof it is enough to find $C_{\ref{lem:exit-exp}},\ep>0$ so that if $h_N(0) \in (a_{k-1},a_{k+1})$ then
\beq\label{eq:tau-k-pr-bd}
P\left(\tau_k^* > \ep C_{\ref{lem:exit-exp}}a_k \right) \le 1-\ep.
\eeq
Indeed, one can deduce from \eqref{eq:tau-k-pr-bd} that 
$$P\left(\tau_k^* > n \ep C_{\ref{lem:exit-exp}} a_k \right) \le (1-\ep )^n,$$
for any positive integer $n$. 
Now the proof finishes by noting that  
$$E[\tau_k^*] \le \sum_{n \ge 1} \ep C_{\ref{lem:exit-exp}}a_k(1-\ep )^n \le C_{\ref{lem:exit-exp}}a_k.$$
Thus it remains to prove \eqref{eq:tau-k-pr-bd}. To do so we use $h_N^2$. If $t<\tau_N(\ctl,0)$ then $i_N(t)\le h_N(t) \le \log N$, and as shown in the proof of Theorem \ref{finiv}, for some constant $\sigma_*^2$, $\sigma^2_t(h_N) = (\sigma_*^2 + o(1))h_N(t)$. Omitting the $t$, from \eqref{fasthdrift} we then have, for constants $c_1,c_2>0$,
\begin{align}\label{eq:h^2-mu}
\mu(h_N^2) &= 2h_N\mu(h_N) + \sigma^2(h_N) \nonumber \\
&= 2c_1z_Ni_Nh_N - 2c_2i_N^2h_N + (\sigma_*^2 + o(1))h_N + O(N^{-1/2}\log N).
\end{align}
Let $c_N^h = 2a_k$, $c_N^Q=N^{-1/6}$ and $L=i_Nh_N$, which is Lipschitz with constant $c_L = O(a_k)$ (to see this expand $i'h' - ih = i'(h'-h) + h(i'-i)$). By the assumptions $k\le m$ and $4^m\delta$ small, $a_k=O(1)$. Using Lemma \ref{aveto0} and noting $(c_N^hc_n^Q)^{1/2} = (N^{-1/6}a_k)^{1/2}=\Omega(N^{-(.15+1/12)})=\omega(N^{-1/4})$, for fixed $T>0$, w.h.p.
\beq\label{eq:h^2-mu-t1}
\sup_{t \le T}\left|\int_0^{t \wedge \tau_k^*} z_N(s)i_N(s)h_N(s) ds\right| = O(a_k^2 N^{-1/12}\log^2 N) = o(a_k).
\eeq
Since $k\le m$, if $t<\tau_k^*$ then $i_N(t) \le h_N(t) \le 2^{m+1}\dlt$ so $2c_2 i_N(t)^2 h_N(t) \le c_2' 4^m \dlt^2h_N(t)$ for some $c_2'$. Thus if $4^m\dlt>0$ is sufficiently small and $t < \tau_k^* \wedge T$, from \eqref{eq:h^2-mu}-\eqref{eq:h^2-mu-t1} we obtain 
$$(h^2(t))^p \ge (\sigma_*^2 - c_2'4^m \dlt^2- o(1))a_{k-1}t \ge (\sigma^2a_{k-1}/2)t,$$
for all large $N$. On the other hand, if $h_N \ge N^{-1/2}$, then $h_N^2$ jumps by $O(N^{-1/2})h_N + O(N^{-1}) = O(N^{-1/2}h_N)$ at rate $O(Nh_N)$. Thus if $t<\tau_k^* \wedge T$ then $\sigma^2(h_N^2(t)) = O(h_N^3(t)) \le c_3 a_{k+1}^3$ for some $c_3>0$. Noting that $a_k=2a_{k-1}=a_{k+1}/2$ and letting $\phi = \sigma_*^2/(64c_3a_k^2)$, if $t<\tau_k^* \wedge T$ then
$$(h_N^2(t))^p - \phi\langle h_N^2\rangle_t \ge (\sigma_*^2/4 - 8\phi c_3 a_k^2)a_kt \ge (\sigma_*^2a_k/8)t.$$
Since for $t<\tau_k^*$, $h_N^2(t) \le a_{k+1}^2\dlt^2 = 4a_k^2\dlt^2$, we find that w.h.p.
$$(h_N^2(t))^p - (h_N^2(t) - h_N^2(0)) - \phi\langle (h_N^2)^m\rangle_t \ge (\sigma_*^2 t/8 - 4a_k\dlt^2)a_k.$$
Letting $t=T=a_k$, $a = \sigma^2 a_k^2/16$ and taking $\dlt \le \sigma_*^2/32$, if $\tau_k^*>t$ then
$$h_N^2(t) - h_N^2(0) -(h_N^2(t))^p \le  -(a+ \phi\langle (h_N^2)^m\rangle_t).$$
Since $\Delta_\star(h_N^2) = O(N^{-1/2}h_N) = O(N^{-1/2}a_k)$,
$$\Delta_\star(h_N^2)\phi = O(N^{-1/2}/a_k) = O(N^{-0.5+0.3}) = o(1),$$
using the fact that $a_k=\Omega(N^{-0.3})$. Using Lemma \ref{lem:sm-est} it follows that 
$$P(\tau_k^* > a_k ) \le 1-\ep$$
with $\ep = 1-2e^{-\phi a}$ and $\phi a = (\sigma_*^2)^2/(2^{10}c_3)$. Now the proof of the claim \eqref{eq:tau-k-pr-bd} follows by setting $C_{\ref{lem:exit-exp}}=1/\ep $. This finishes the proof of the lemma.
\end{proof}

Next, we estimate the exit probabilities.

\begin{lemma}\label{lem:exit-pr}
Fix a positive integer $m>0$ and $\ep>0$. Let $\dlt,\tau_k,\tau_k^*$ be as in Lemma \ref{lem:exit-exp} and its proof. Assume $h_N(0) = a_k + O(N^{-1/2}) \ge N^{-0.3}$. Then for large $N$ and all $k\le m$,
 $$P(h_N(\tau_k^*) \ge a_{k+1}) \le 1/3 + \ep.$$
\end{lemma}

\begin{proof}
Using the facts that $\tau_k$ is the exit time from $(a_{k-1},a_{k+1}) = (a_k/2,2a_k)$, $\tau_k^* \le \tau_k$ and $h_N$ jumps by $O(N^{-1/2})$, for any $T>0$
\begin{align*}
E[h_N(\tau_{k^*} \wedge T)] &\ge 2a_kP(h_N(\tau_k^* \wedge T) \ge 2a_k) + (a_k/2)(1-P(h_N(\tau_k^* \wedge T) \ge 2a_k)) - O(N^{-1/2}). 
\end{align*}
Rearranging and noting $N^{-1/2}=o(a_k)$,
\begin{align}\label{eq:h-hit-est}
\frac{3}{2}P(h_N(\tau_k^* \wedge T) \le 2a_k) \le \frac{1}{a_k}E[h_N(\tau_k^* \wedge T)] - \frac{1}{2} + o(1).
\end{align}
As with the derivation of \eqref{eq:h^2-mu}, for $t<\tau_k^*$,
$$\mu_t(h_N) = c_1z_N(t)i_N(t) - c_2i_N(t) + O(N^{-1/2}\log N),$$
and so
\beq\label{eq:h_t-exp}
E[h_N(\tau_k^* \wedge T)] - E[h_N(0)] = E\left[\int_0^{\tau_k^* \wedge T} (c_1z_N(s)i_N(s) - c_2 i_N(s)^2 + O(N^{-1/2}\log N))ds \right].
\eeq
Using Lemma \ref{aveto0} with $L=i_N$, $c_N^h=2a_k$ and $c_N^Q = N^{-1/6}$ and noting as in the previous proof that $(c_N^hc_N^Q)^{1/2}=\omega(N^{-1/4})$, 
\[
\int_0^{\tau_k^* \wedge T} z_N(s) i_N(s) \, ds = O\left(a_kN^{-1/12}\log^2 N\right) = o(a_k).
\]
Now, the second term on the RHS of \eqref{eq:h_t-exp} is negative and since $a_k=\Omega(N^{-0.3})$ the third term is $o(a_k)$. So $E[h_N(\tau_k^* \wedge T)] \le E[h_N(0)] + o(a_k) = a_k(1+o(1))$, noting that $h_N(0) = a_k+O(N^{-1/2}) = a_k + o(a_k)$. Combining with \eqref{eq:h-hit-est},
\begin{align*}
P(h_N(\tau_k^* \wedge T) \ge 2a_k)  \le (2/3)(1 - 1/2 + o(1))  \le 1/3 + o(1).
\end{align*}
Using Lemma \ref{lem:exit-exp} we see that we can take $T>0$ large enough so that $P(\tau_k^* > T) )\le \ep/2$, uniformly for $k\le m$. Combining with the above display gives the desired result.
\end{proof}

\vskip10pt
Equipped with Lemma \ref{lem:exit-exp} and Lemma \ref{lem:exit-pr} we now prove Lemma \ref{lem:Hdn1}.

\begin{proof}[Proof of Lemma \ref{lem:Hdn1}]
We first recall the following two facts about a simple random walk on $\Z$ with probability $p<1/2$ of increasing by 1 and $(1-p)$ of decreasing by 1, at each time step.
\begin{enumerate}[noitemsep]
 \item Starting from $0$ the probability to ever reach $k>0$ is $(p/(1-p))^k$, and  
 \item starting from $k$, the expected number of jumps out of $k$ is equal to  $1/(1-2p)$.
\end{enumerate}

Fix a positive integer $m$ and suppose $h_N(0) \le \dlt = a_0$. Using Lemma \ref{lem:exit-pr} and by comparison it follows that, uniformly for $k\le m$ 
\begin{enumerate}[noitemsep]
 \item For any $k > 0$, the probability that $h_N$ reaches $[a_k,\infty)$ is at most $ (1/2 + o(1))^k$. 
 \item If $h_N(t)=a_k+O(N^{-1/2})$, the expected number of times we perform \\ the step of exiting $(a_{k-1},a_{k+1})$ before time $\tau_N(\ctl,0)$ is at most $3+o(1)$.
\end{enumerate}
Using point 2 and summing over the expected number of exits from each level (and adding one for the initial exit, since $h_N(0)$ may not be equal to $a_k+O(N^{-1/2})$ for some $k$) to find that
\begin{align*}
E[\tau_N(\ctl,0) \wedge \tau^+_{2^m\dlt}(h_N)] &\le C_{\ref{lem:exit-exp}}( a_0 + 3\sum_{k=-\infty}^0 a_k + 3\sum_{k=1}^m(1/2+o(1))^ka_k \\
&= C_{\ref{lem:exit-exp}}\dlt ( 1 + 3 + C_m) = D_m\dlt
\end{align*}
where $C_m,D_m$ are constants that depend only on $m$.  Using Markov's inequality,
$$P(\tau_N(\ctl,0) \wedge \tau^+_{2^m\dlt}(h_N) \ge MD_m\dlt) \le 1/M.$$
Let $M=2/\eps$ so the above probability is at most $\eps/2$. Since $\tau_N(\ctl,0) \le \tau_{N^{-0.3}}^-(h_N)$, using point 1, $P(\tau_N(\ctl,0) \ge \tau^+_{2^m\dlt}(h_N)) \to 0$ as $m\to\infty$, so take $m$ large enough that this probability is at most $\ep/2$. Combining,
$$P(\tau_N(\ctl,0) \ge 2D_m\dlt/\eps \wedge \tau_{2^m\dlt}(h_N)) \le \eps.$$
Take $\dlt>0$ small enough that $2D_m\dlt/\eps \le \eps$. Since $\tau_{2^m\dlt}^+(h_N) \le \tau_{\log N}^+(h_N)$ for large $N$, the result is proved.
\end{proof}

\vskip10pt
Next we prove Lemma \ref{lem:Hdown}.

\vskip10pt

\begin{proof}[Proof of Lemma \ref{lem:Hdown}]
Recall that $\tau = \inf\{t:H_t^N =0 \ \text{or} \ H_t^N \ge N^{0.24}\}$ and define
$$\tau' = \tau \wedge \tau^+_{C_{\ref{maxz}}\sqrt{\log N}}(|z^N|).$$
Note that $H_0^N \le N^{1/5}$ and that $\tau' \wedge N^{1/4} \ge \tau \wedge N^{1/4}$ w.h.p.~by assumption. We will show that i) $\tau' \le N^{1/4}$ w.h.p.~and that ii) $H_{\tau'} < N^{.24}$ w.h.p. To deduce the result from these, first combine i) with the second assumption to find that w.h.p.
$$\tau \wedge N^{1/4} \le \tau'\wedge N^{1/4} = \tau',$$
and that since $\tau' \le \tau$, w.h.p.~$\tau=\tau' \le N^{1/4}$. Then, note that $H_{\tau}$ is either $\ge N^{.24}$ or is equal to $0$, so that if $\tau=\tau'$ and $H_{\tau'} < N^{.24}$, then $H_{\tau}=0$.\\

\noindent\textbf{Showing that $\tau' \le N^{1/4}$ w.h.p.} The idea is to approximate $(I^N,J^N,K^N)$ by a multi-type continuous time branching process. Such processes are characterized by having transition rates that are a homogeneous linear function of the variables. In $(I^N,J^N,K^N)$ there are two non-linear interactions: $I+I$ and $S+I$ partnership formation. If $t<\tau'$ then since $I^N_t \le H^N_t \le N^{.24}$, the rate at which a pair of single $I$ form a partnership is $O((I^N)^2/N) = O(N^{0.48-1}) = o(N^{-1/4})$. Therefore, with high probability no such events occur on the interval $[0,\tau' \wedge N^{1/4}]$. $S+I$ partnerships form at rate $O(S^NI^N/N)$. If $t<\tau'$ then using the bound on $|z^N|$ (and omitting $t$),
\begin{align*}
S^NI^N/N= (Y^N-I^N)I^N/N &= y_* I^N + z^N I^N/\sqrt{N} - (I^N)^2/N \\
&= y_*I + O(\sqrt{\log(N)/N}I^N) + o(N^{-1/4}) \\
&= y_*I^N + o(N^{-1/4}).
\end{align*}
Thus, if we pretend the rate is $y_*I^N$ (i.e., generate transitions using two independent sources of randomness, one with rate $y_*I^N$ and the other with rate $o(N^{-1/4})$), then w.h.p.~the process so obtained will be identical to the original process on the time interval $[0,\tau' \wedge N^{1/4}]$.\\

Thus, the continuous time three-type (Markov) branching process $(\I,\J,\K)$ obtained by letting $(\I_0,\J_0,\K_0) = (I^N_0,J^N_0,K^N_0)$ and ignoring $I+I$ transitions and the non-linear part of $S+I$ transitions is such that
$$P((\I_t,\J_t,\K_t) \ne (I^N_t,J^N_t,K^N_t) \ \text{for some} \ t \le \tau' \wedge N^{1/4}) = o(1).$$
If $(I^N,J^N,K^N)=(0,0,0)$ then $t \ge \tau'$, so it would be enough to show the extinction time of $(\I,\J,\K)$ is $o(N^{1/4})$ w.h.p. We first extract an embedded one-type CMJ (non-Markov) branching process. To do this, we note the following two points:
\begin{enumerate}
\item \textit{Initial decay:} each initial particle of type $\J,\K$ decays at rate $\ge r_-$ (regardless of its type) into $0,1$ or $2$ type $\I$ particles before ever producing additional particles of other types, and
\item \textit{Reproduction cycle:} each type $\I$ particle follows the evolution described in Figure \ref{fig:critval}, \\ yielding $0,1$ or $2$ type $\I$ particles upon reaching the set $\{D,E,F,G\}$.
\end{enumerate}
Since there are initially $O(N^{0.2})$ particles, by point 1., with high probability, within constant times $\log N$ time every initial particle of type $\J,\K$ has turned into $0,1$ or $2$ type $\I$ particles. Thus, since $\log N = o(N^{1/4})$, we may assume all initial particles have type $\I$.\\

Point 2. says that we can use the Markov chain described in Figure \ref{fig:critval} to determine the timing and number of type $\I$ offspring of each type $\I$ particle. The offspring distribution has mean $R_0=1$ and is supported on the set $\{0,1,2\}$. Referring to Figure \ref{fig:critval}, the waiting time to produce offspring is the absorption time at $\{D,E,F,G\}$ starting from $A$, which is at most exponential($1 + r_+y_*$) + exponential($r_-$).\\

In the embedded one-type process of type $\I$ particles, the set of descendants of any particle forms a critical Galton-Watson tree. We recall a couple of facts that hold for such trees, when the offspring distribution has finite variance. The height of the tree (maximum distance to the root) is greater than or equal to $n$ with probability $O(1/n)$, and the total number of vertices is greater than or equal to $n$ with probability $O(n^{-1/2})$. Thus with $O(N^{0.2})$ initial particles, with high probability the tallest tree has height $O(N^{0.22})$ and the sum of tree sizes is $O(N^{0.44})$. In particular there are in total $O(N^{0.44})$ offspring production events. Since the waiting time for offspring is at most the sum of two exponential random variables with fixed constant rates, the longest waiting time for offspring is whp bounded by constant times $\log N$. Bounding the waiting times by their maximum and noting the height bound, whp the process dies out within $O(N^{0.22}\log N) = o(N^{1/4})$ amount of time.\\

\noindent\textbf{Showing that $H_{\tau'} < N^{.24}$ w.h.p.} Using the bound on $|z^N|$ and the fact that $\gamma/2 < \eta$ and $I^N \le H^N$, from \eqref{Hdrift} we find that for $t < \tau'$,
\beq
\label{eq:muH_ubd}
\mu_t(H^N) \le \frac{a_1}{2}\left(\sqrt{\frac{\log N}{N}} + O(1/N)\right) I^N_t \le a_1 \sqrt{\frac{\log N}{N}} H^N_t
\eeq
for some constant $a_1>0$. For $b>0$ to be determined, letting $f(x) = e^{-bx}$ and using Taylor's theorem, for fixed $x,y$ we get
$$f(y)-f(x) = (-b(y-x) + \frac{b^2}{2}(y-x)^2)e^{-bx} - \frac{b^3(y-x)^3}{3!}e^{-bz}$$
where $z$ is between $x$ and $y$. Applying this to $e^{-bH^N}$, since $\Delta_\star(H^N)=O(1)$ we will have $e^{-bz}=e^{-b(x+O(1))}$ and $(y-x)^3=O(1)$ across jump times of $H^N$. Noting that transition rates are bounded by $O(H^N)$ and multiplying by those transition rates, upon summing over $y$ we obtain
$$\mu(e^{-bH^N}) = \left(-b\mu(H^N) + \frac{b^2}{2}\sigma^2(H^N) + O(b^3H^Ne^{O(b)})\right)e^{-bH^N}.$$
Since $\sigma^2(H^N) = \Omega(H^N)$, applying \eqref{eq:muH_ubd} we obtain
$$\mu(e^{-b H^N}) \ge \left(-a_1 \sqrt\frac{\log N}{N} + \Omega(b) + O(b^2e^{O(b)})\right)bH^Ne^{-bH^N}.$$
Taking $b = N^{-0.22}$ and noting that $b^2e^{O(b)} = o(b)$ we deduce $\mu(e^{-bH^N})$ is non-negative for large $N$, so the process $\xi_t = e^{-bH^N_{t \wedge \tau'}}$ is a submartingale. From the definition of $\tau'$, the fact that the jumps are of size $O(1)$ and the fact that $e^{-bH^N} \le 1$, we use the optional stopping theorem to obtain
$$\begin{array}{cc}
& P(H_{\tau'} \ge N^{0.24})\exp(-b(N^{0.24} + O(1))) + P(H_{\tau'}<N^{.24}) \\
& \ge E[e^{-bH_{\tau'}}] \ge E[e^{-bH_0}].
\end{array}$$
Since $H_0 \le N^{0.2}$, $bH_0 = o(1)$ and $bN^{0.24} = \omega(1)$, so
$$P(H_{\tau'} < N^{.24}) \ge e^{-o(1)} - e^{\omega(1)} = 1 - o(1).$$

\end{proof}

\clearp

\section{Computing $R_0$} \label{sec:Rzero}

In this section we show that \eqref{detA0} is equivalent to the condition $R_0=1$. To do this we recall \eqref{r0eq}, namely the definition of $R_0$:
$$
R_0 = P_A(X_\tau = F) + 2 P_A(X_\tau = G)
$$
where $\tau$ is the hitting time of $\{D,E,F,G\}$ for the Markov chain with rates drawn in Figure \ref{fig:critval}.
To calculate $R_0$ we let
$$
f(x) = P_x(X_\tau = F) + 2 P_x(X_\tau = G)
$$
and note that $f(D)=0$, $f(E)=0$, $f(F)=1$, and $f(G)=2$. By considering what happens on the first jump from each state we see that
\begin{align}
f(A) & = \frac{ r_+ y_*}{1+ r_+ y_*} f(B),
\label{ha}\\
f(B) & = \frac{\lambda}{\lambda + 1 + r_-} f(C) + \frac{r_-}{\lambda + 1 + r_-} \cdot 1,
\label{hb}\\
f(C) & = \frac{2}{2 + r_-} f(B) + \frac{r_-}{2 + r_-} \cdot 2.
\label{hc}
\end{align}
The equations \eqref{hb} and \eqref{hc} can be rewritten as
\begin{align*}
\frac{\lambda + 1 + r_-}{\lambda} f(B) &= f(C) + \frac{r_-}{\lambda}, \\
- \frac{2}{2 + r_-} f(B) &= - f(C)+ \frac{2r_-}{2 + r_-}.
\end{align*}
Adding these last two equations we have
$$
\left( \frac{\lambda + 1 + r_-}{\lambda} - \frac{2}{2 + r_-} \right) f(B) = r_- \left( \frac{1}{\lambda} +  \frac{2}{2 + r_-} \right).
$$
Adding the fractions in the parentheses we have
$$
\frac{(2+r_-)(\lambda + 1 + r_-) - 2\lambda}{\lambda (2+r_-)} \cdot  f(B) = r_- \frac{r_- + 2 + 2\lambda}{\lambda (2+r_-)}.
$$
Therefore we deduce
$$
f(B) = r_- \frac{r_- + 2 + 2\lambda}{(2+r_-)(\lambda+1 + r_-) - 2\lambda}
=  r_- \frac{r_- + 2 + 2\lambda}{2 + (3 + \lambda) r_- + r_-^2}
$$
where we have used the simplification of the denominator used in going from \eqref{detA} to \eqref{detA0}.
Using \eqref{ha} now we have
$$
f(A) =\frac{ r_+ y_*}{1+ r_+y_*} \cdot  r_- \frac{r_- + 2 + 2\lambda}{2 + (3+ \lambda) r_- + r_-^2}
$$
The expression on the RHS above will be equal to 1 when
$$
(r_+ y_* +1)(2 + (3+ \lambda) r_- + r_-^2) = r_- (r_+ y_*)(r_- + 2 + 2\lambda)
$$
which is the same as \eqref{detA0}.

When $R_0=1$, $f(A)=1$, so $f(B)=(1+r_+ y_*)/r_+ y_*$ and
\beq
f(C) = \frac{2r_-}{2+r_-} = \frac{2}{2+r_-} \cdot f(B)
\label{hC}.
\eeq

\section*{Appendix}

Here we prove any uncited results from Section \ref{sec:sampath}.

\begin{proof}[Proof of Lemma \ref{lem:taylor}]
By It{\^o}'s lemma \cite[Theorem I.4.57]{jacod},
$$f(X_t) = f(X_0) + (f'(X_-) \cdot X)_t + \frac{1}{2}(f''(X_-) \cdot \lng X^c \rng)_t + \sum_{s \le t}(f(X_s)-f(X_{s^-})-f'(X_{s^-})\Delta X_s),
$$
where $X^c$ is the continuous part of $X^m$. Since $X$ has bounded jumps, both $X$ and $\lng X^c \rng$ are locally integrable. Furthermore, since $f,f',f''$ are continuous, both are locally bounded, so since $X$ is locally bounded, $f(X),f'(X_-),f''(X_-)$ are locally integrable. Taking the compensator of both sides (and noting that $\lng X^c \rng$ is its own compensator),
$$f(X_t)^p - (f'(X_-) \cdot X^p)_t = \frac{1}{2}(f''(X_-)\lng X^c \rng)_t + \bigg(\sum_{s \le t}(f(X_s)-f(X_{s^-})-f'(X_{s^-})\Delta X_s)\bigg)^p.
\label{eq:taylor-1}
$$
To obtain the result we take the derivative, but we need to estimate the last term more carefully. By a Taylor expansion, the term under the sum is at most \\$\frac{1}{2}(\Delta X_s)^2\sup_{|x-X_{s^-}|\le \Delta_\star(X)}|f''(x)|$. Thus
$$\big|\sum_{t < s \le t+h}f(X_s)-f(X_{s^-})-f'(X_{s^-})\Delta X_s \big|\le \frac{R(h)}{2}\sum_{t<s \le t+h}(\Delta X_s)^2,$$
where
$$R(h) = \sup\{|f''(x)| \colon |x-X_{s^-}| \le \Delta_\star(X) \ \text{for some} \ t < s \le t+h\}.$$
Moreover,
$$\bigg(\sum_{t<s \le t+h}(\Delta X_s)^2 \bigg)^p = \lng X^d \rng_{t+h} - \lng X^d \rng_t$$
where $X^d$ is the discontinuous part of $X^m$. Since $X$ is right-continuous,
$$\lim_{h\to 0^+}R(h) = \sup_{|x-X_t| \le \Delta_\star(X)}|f''(x)|.$$
Let $Q_t$ denote the right-hand side of \eqref{eq:taylor-1}. Since $\lng X \rng = \lng X^c \rng + \lng X^d \rng$, it follows that
$$|Q_{t+h}-Q_t| \le \frac{1}{2}R(h) (\lng X \rng_{t+h} -\lng X \rng_t).$$
Then, dividing by $h$ and letting $h\to 0$ and referring again to \eqref{eq:taylor-1} we obtain the desired result, since $\frac{d}{dt}f(X_t)^p = \mu_t(f(X)), \frac{d}{dt}X_t^p = \mu(X), \frac{d}{dt}\lng X \rng_t = \sigma_t^2(X)$, and $X_{t^-}=X_t$ for a.e.$\,t$, since $\{t\colon\Delta X_t\ne 0\}$ is countable.
\end{proof}

\begin{proof}[Proof of Lemma \ref{lem:driftbar}]
Suppose $|X_0 - x/2| \leq \Delta_\star(X)/2$. Let $\tau = \inf\{t:|X_t - x/2| \geq x/2\}$. If $t \le \tau$ and $\phi>0$ then it follows that
  $$X_t^m - \phi \langle X\rangle_t \geq X_t - X_0 + (\mu_\star - \phi \sigma_\star^2)t.$$
  Let $\phi = \min(\mu_\star/\sigma_\star^2,1/(2\Delta_\star(X)))$. Hence, $X_t^m - \phi \langle X \rangle_t \ge X_t -X_0$ for any $t \le \tau$. Next set $a=(x-\Delta_\star(X))/2$. Therefore, $X_{\tau} \geq x$ implies that $X_\tau-X_0 \ge a$. Since $\phi\Delta_\star(X) \le 1/2$, $\phi \ge \mu_\star/2C_\Delta\sigma_\star^2$, and $a \ge x/4$, we can apply Lemma \ref{lem:sm-est} to conclude that
 \beq\label{eq:X_tau-bd}
 P(X_{\tau} \geq x \ \mid \ |X_0 - x/2| \leq \Delta_\star(X)/2) \le 2\exp(-\mu_\star x / 8C_\Delta\sigma^2_\star).
 \eeq
  Since $|\mu_t(X)| \le C_{\mu_\star}$ for $t <\tau$ and $|X_{\tau} - X_0| \ge (x-\Delta_\star(X))/2$, it follows that
  $$|X_\tau^m|  - \phi \langle X\rangle_\tau \ge (x-\Delta_\star(X))/2 - (C_{\mu_\star}+\phi\sigma_\star^2) \tau \ge (x-\Delta_\star(X))/2 - 2C_{\mu_\star} \tau,$$
 where the last step follows from the fact that 
 $\phi \sigma^2_\star \le \mu_\star \le C_{\mu_\star}$. 
Using the same value of $\phi$, but changing the value of $a$ to $a = (x-\Delta_\star(X))/4 \ge x/8$, let $t = (x-\Delta_\star(X))/8C_{\mu_\star} \ge x/16C_{\mu_\star}$. Then $\tau \le t$ implies that $|X_\tau^m| - \phi \langle X \rangle_{\tau} \ge a$. Therefore, using Lemma \ref{lem:sm-est} we further deduce
\beq\label{eq:tau-bd}  
 P(\tau \le x/16C_{\mu_\star} \ \mid \ |X_0 - x/2| \leq \Delta_\star(X)/2) \le 2\exp(-\mu_\star x / 16C_\Delta\sigma^2_\star).
 \eeq
Thus taking a union bound, from \eqref{eq:X_tau-bd}-\eqref{eq:tau-bd}, we deduce
\beq\label{eq:X-tau-tau-bd}
P(\tau \leq x/16C_{\mu_\star} \quad\hbox{or}\quad X_{\tau} \geq x \ \mid \ |X_0 - x/2| \leq \Delta_\star(X)/2) \leq 4\exp(-\mu_\star x/16C_\Delta\sigma^2_\star).
\eeq
Iterating the estimate $\lfloor \Gamma \rfloor$ times, alternately stopping the process when $|X_t - x/2| \leq \Delta_\star(X)/2$ and $|X_t- x/2| \geq x/2$, the result follows from a union bound.  

Indeed, setting $\tau_0=0$, 
\[
\tau_{2j-1}=\inf\{ t >\tau_{2j-2}: |X_t -x/2| \le \Delta_\star(X)/2\},
\]
and
\[
\tau_{2j}=\inf\{ t >\tau_{2j-1}: |X_t -x/2| \ge x/2\},
\]
for $j=1,2,\ldots,$ from \eqref{eq:X-tau-tau-bd} we see
\[
P(X_{\tau_{2j}} \ge x \text{ or } \tau_{2j} - \tau_{2j-1} \le  x/16C_{\mu_\star} | \tau_{2j-1} <\infty) \le 4\exp(-\mu_\star x/16C_\Delta\sigma^2_\star).
\]
Hence, taking a union bound, as mentioned above, (note that $\tau_{2j-1}=\infty$ for some $j$ automatically implies $\sup_t X_t \le x$) we see
\[
P(X_{\tau_{2j}} \ge x \text{ for some } j \le \lfloor \Gamma \rfloor, \text{ or } \tau_{2 \lfloor \Gamma \rfloor}  \le \lfloor \Gamma \rfloor x/16C_{\mu_\star} | X_0 \le x/2) \le 4\lfloor \Gamma \rfloor\exp(-\mu_\star x/16C_\Delta\sigma^2_\star).
\]
Since $\Delta_\star(X) <x$ it can be easily checked that if $X_{\tau_{2j}} < x$  for all $j \le \lfloor \Gamma \rfloor$ then $X_t < x$ for all $t \le \tau_{2\lfloor \Gamma \rfloor}$. Now the desired probability bound is immediate. This completes the proof.
\end{proof}

\begin{proof}[Proof of Corollary \ref{cor:driftbar}]
We use the following ``continuation trick''. Define a new process $\tilde X$ by
$$\tilde X_t = X_{t\wedge \tau} - \mu_\star(t - t\wedge \tau).$$
In words, $\tilde X$ is equal to $X$ up to time $\tau$, at which point it decreases at a fixed deterministic speed $\mu_\star$. Since $t\mapsto \tilde X_t$ is continuous on $[\tau,\infty)$, $\Delta_\star(\tilde X) \le \Delta_\star(X)$. Moreover, it is easy to check that $\tilde X$ satisfies \eqref{eq:driftbar-est} assuming only $0<\tilde X_t<x$ (i.e., without assuming $t<\tau$). Thus, Lemma \ref{lem:driftbar} applies to $\tilde X$. Since $X_t=\tilde X_t$ for $t\le\tau$, the result follows.
\end{proof}

\begin{proof}[Proof of Lemma \ref{lem:limproc}]
We show the conditions (4.1)-(4.7) of Theorem 4.1 of \cite[Chapter 7]{EthierKurtz} are satisfied. The fact that $a$ and $b$ are Lipschitz ensures the martingale problem for the limit process is well-posed. $(X^N)^p$ and $\langle X^N \rangle$ play the role of $B_N$ and $A_N$ respectively. Since $v^{\top}\langle X^N \rangle_tv = \langle v^{\top} X^N \rangle_t$ for vector $v$ and the latter is $\R$-valued, the process $\langle X^N \rangle$ is non-negative definite. We also note that $X_i-X_i^p$ and $X_iX_j - \langle X \rangle_{ij}$ are local martingales as required by (4.1) and (4.2).  

Since the jump size of $X^N$ tends to $0$ (4.3) is satisfied. Applying \cite[Lemma I.4.24]{jacod} we see that jump size of $X^p$ and $\langle X \rangle$ also converge to $0$, which gives (4.4) and (4.5). The requirements of (4.6) and (4.7) are immediate from \eqref{eq:diff-bd-1}-\eqref{eq:diff-bd-2} above. The rest is straightforward, so we omit the details.
\end{proof}

\end{document}